\documentclass[amstex,12pt]{amsart}
\usepackage{amssymb}
\usepackage{amsmath}
\usepackage{amscd}
\usepackage{mathrsfs}
\usepackage{graphicx}
\usepackage{latexsym}

\theoremstyle{definition} 

\newtheorem*{definition}{Definition} 
\newtheorem*{remark}{Remark}

\theoremstyle{plain}  

\newtheorem{theorem}{Theorem}[section]
\newtheorem{proposition}[theorem]{Proposition}
\newtheorem{corollary}[theorem]{Corollary}
\newtheorem{lemma}[theorem]{Lemma}

\DeclareMathOperator{\Stab}{Stab}

\DeclareMathOperator{\Fix}{Fix}

\DeclareMathOperator{\Isom}{Isom}

\DeclareMathOperator{\diam}{diam}
\DeclareMathOperator{\id}{id}
\DeclareMathOperator{\almostall}{\mbox{\rm{a.e.}}}

\DeclareMathOperator{\R}{\mathbb R}
\DeclareMathOperator{\N}{\mathbb N}
\DeclareMathOperator{\Z}{\mathbb Z}
\DeclareMathOperator{\H2}{\mathbb H}
\DeclareMathOperator{\S1}{\mathbb S}

\begin{document}

\title[Normalizer, divergence type, and Patterson measure]
{Normalizer, divergence type, and Patterson measure \\ for discrete groups of the Gromov hyperbolic space}

\author{Katsuhiko Matsuzaki}
\address{Department of Mathematics, School of Education, Waseda University,\endgraf
Shinjuku, Tokyo 169-8050, Japan}
\email{matsuzak@waseda.jp}

\author{Yasuhiro Yabuki}
\address{Tokyo Metropolitan College of Industrial Technology,\endgraf
Arakawa, Tokyo 116-0003, Japan}
\email{yabuki@metro-cit.ac.jp}

\author{Johannes Jaerisch}
\address{Graduate School of Mathematics, Nagoya University,\endgraf
Furocho, Chikusaku, Nagoya 464-8602, Japan}
\email{jaerisch@math.nagoya-u.ac.jp}

\makeatletter
\@namedef{subjclassname@2010}{%
\textup{2010} Mathematics Subject Classification}
\makeatother

\thanks{This work was supported by JSPS KAKENHI 25287021.}

\begin{abstract}
For a non-elementary discrete isometry group $G$ of divergence type acting on a proper geodesic
$\delta$-hyperbolic space, we prove that its Patterson measure is quasi-invariant under the normalizer of $G$.
As applications of this result, we have: (1) under a minor assumption, such a discrete group $G$ admits no proper conjugation,
that is, if the conjugate of $G$ is contained in $G$, then it  coincides with $G$; (2) the critical exponent of 
any non-elementary normal subgroup of $G$ is strictly greater than half of that for $G$.
\end{abstract}

\maketitle

\section{Introduction}
The Patterson--Sullivan theory of Kleinian groups studies the dynamics and geometry of discrete isometry groups of 
the hyperbolic space $\H2^{n+1}$ or $(n+1)$-dimensional
complete hyperbolic manifolds using invariant conformal measures on the boundary $\S1^n$ at infinity
(see \cite{S1,S2}). Recently, they have often been generalized to 
simply connected Riemannian manifolds with variable negative curvature bounded above or,
more generally, to ${\rm CAT}(-1)$-spaces (see \cite{R1}). 
Following the great success of this theory, it was extended to discrete groups acting on other metric spaces of hyperbolic nature and their boundary at infinity. 
The Gromov hyperbolic space is a typical object to which the theory
of the classical hyperbolic space is generalized; in fact, the Patterson--Sullivan theory was developed 
for discrete isometry groups on a proper geodesic
$\delta$-hyperbolic space by Coornaert \cite{Co}.

Among other important results on Kleinian groups in this field, the investigation of normal subgroups of
a Kleinian group (equivalently, normal covers of a hyperbolic manifold) has considerably progressed. 
For instance, the characterization of the amenability of the covering of a convex compact manifold
in terms of certain geometric invariants has been proved.
This was originally due to Brooks, and a recent account oriented toward the Patterson--Sullivan theory can be found in \cite{RT}.
The Patterson measure is the characteristic invariant conformal measure of a Kleinian group
and its invariance under the normalizer
was shown by the present authors \cite{MY1} in the following form. The survey article \cite{P} also explained
a consequence of this theorem and a relation to the problem of the critical exponents of normal subgroups. 

\begin{theorem}\label{kleinian}
Let $\Gamma$ be a non-elementary Kleinian group acting on the hyperbolic space $\H2^{n+1}$ whose
Poincar\'e series diverges at the critical exponent. Then, the Patterson measure for $\Gamma$
is invariant under the normalizer $N(\Gamma)$ of $\Gamma$.
\end{theorem}

In this study, we generalize this theorem to a discrete isometry group $\Gamma$ of a proper geodesic
$\delta$-hyperbolic space $(X,d)$. As a counterpart to the conformal invariant measure,
the quasiconformal measure of quasi-invariance was introduced in \cite{Co}. Roughly speaking, a concept on the usual geometry is defined on the Gromov hyperbolic space with controllable ambiguity; thus,
the notion of invariance of a conformal measure must be appropriately weakened. We say that
an $s$-dimensional quasiconformal measure $\mu$ on the boundary $\partial X$ is $\Gamma$-quasi-invariant if
there is a constant $D \geq 1$ independent of $\gamma \in \Gamma$ such that the Radon--Nikodym derivative of
the pull-back $\gamma^*\mu$ to $\mu$ is comparable with the $s$-dimensional magnification rate of $\gamma$
with multiplicative error factor $D$.

On the contrary,  the critical exponent of a discrete isometry group
$\Gamma \subset \Isom(X,d)$ is determined in exactly the same way 
as the exponential growth rate of the orbit, and in both classical and modern cases, a conformal measure or
a quasiconformal measure of the dimension at the critical exponent reflects the geometry of $\Gamma$.
This is what we defined as the Patterson measure. Moreover,
the divergence of the Poincar\'e series at the critical exponent
is a distinguished property for $\Gamma$ and ensures uniqueness of the Patterson measure in a certain sense.
If $\Gamma$ satisfies this property, then $\Gamma$ is said to be of divergence type.

The main theorem of this study is the following. It will be proved in Section \ref{6}. 

\begin{theorem}\label{i-main}
Let $\Gamma \subset \Isom(X,d)$ be a non-elementary discrete group of divergence type acting on a proper geodesic
$\delta$-hyperbolic space $(X,d)$. 
Then, the Patterson measure for $\Gamma$ is quasi-invariant under the normalizer $N(\Gamma)$ of $\Gamma$.
\end{theorem}

We present two applications of this theorem in Sections \ref{7} and \ref{8}.
For Kleinian groups, we used Theorem \ref{kleinian} for a problem of proper conjugation in \cite{MY1} and
for a new proof of the theorem by Roblin \cite{R2} concerning the lower bound of the critical exponents of normal subgroups in \cite{J}.
Our applications correspond to these results.

\begin{theorem}\label{i-proper}
Let $G \subset \Isom(X,d)$ be a non-elementary discrete isometry group of divergence type
that acts on $X$ uniformly properly discontinuously.
If $\alpha G \alpha^{-1} \subset G$ for $\alpha \in \Isom(X,d)$, then $\alpha G \alpha^{-1} = G$.
\end{theorem}

When $G$ is quasi-convex cocompact, the same conclusion was proved in \cite{MY2}. Theorem \ref{i-proper} is an extension of
this case. For technical reasons, we assume a certain uniformity property of the properly discontinuous action.
As we mentioned in our previous paper, we can apply this theorem to the problem of proper conjugation 
for a subgroup $G$ of any hyperbolic group. 

\begin{theorem}\label{i-lower}
Let $G \subset \Isom(X,d)$ be a discrete group of divergence type and let
$\Gamma \subset G$ be a non-elementary normal subgroup.
Then, the critical exponent of $\Gamma$ is strictly greater than half of that for $G$.
\end{theorem}

The arguments in \cite{J} for Kleinian groups require
only basic geometry on the hyperbolic space $\H2^{n+1}$ except those for showing the strict inequality.
The basic geometric properties can be adjusted 
to discrete isometry groups of the Gromov hyperbolic space. 
To obtain the strict inequality, one should assume that $G$ is of divergence type and 
apply Theorem \ref{i-main}.

The fundamental fact for proving Theorem \ref{i-main} is that the Patterson measure for a discrete group $\Gamma$
has a certain uniqueness property if $\Gamma$ is of divergence type. We call this quasi-uniqueness and it
is formulated such that
if $\mu$ and $\mu'$ are two Patterson measures for $\Gamma$, then they are mutually absolutely continuous, and
the Radon--Nikodym derivative $d\mu'/d\mu$ is bounded from above and away from zero almost everywhere on $\partial X$.
This follows from the ergodicity of the action of $\Gamma$ on $\partial X$ with respect to the Patterson measure,
as was shown in \cite{Co}. 

However, we have to obtain more explicit bounds in terms of the quasi-invariance constants $D$
and $D'$ for $\mu$ and $\mu'$ as well as their total masses, which is presented in Section \ref{5}.
This is due to the fact that it is not sufficient for Theorem \ref{i-main} that a measure $\mu'$ given by the pull-back of $\mu$ 
under $g \in N(\Gamma)$
is also a Patterson measure for $\Gamma$. 

For the quasi-invariance
under $N(\Gamma)$, we must establish the uniformity of the bounds of $d\mu'/d\mu$
independent of $g \in N(\Gamma)$.
Moreover, to estimate the total mass of $\mu'$, which is the total mass of the quasiconformal measure $\mu$
with the reference point changed by $g \in N(\Gamma)$, we take $\mu$ as the Patterson measure obtained by
the canonical construction from the Poincar\'e series of $\Gamma$. The advantage of this construction is that
the invariance of the Poincar\'e series under the normalizer $N(\Gamma)$ is reduced to $\mu$; hence, we see that
the total mass of $\mu'$ is comparable with that of $\mu$. This is an idea for the proof of Theorem \ref{i-main}
that is presented in Section \ref{6}.

In the next three sections (2--4), we carry out preliminary work toward the main theorem.
The fact that a discrete group $\Gamma$ of divergence type acts on $\partial X$ ergodically with respect to the Patterson measure $\mu$ 
is a consequence of the condition that $\mu$ has positive measure on the conical limit set $\Lambda_c(\Gamma) \subset \partial X$.
These are well-known arguments for Kleinian groups. 
In fact, this fact originates in the Hopf--Tsuji problem for Fuchsian groups and the Lebesgue measure. 
Sullivan \cite{S1, S2} generalized this to
Kleinian groups and their Patterson measures by considering the geodesic flow on the hyperbolic manifold $\H2^{n+1}/\Gamma$.
Later, Tukia \cite{T1} presented an elementary proof without the argument of the geodesic flow.
One can expect that his proof is applicable to discrete isometry groups of the Gromov hyperbolic space if 
necessary changes are made. We do this in Section \ref{3}, where Tukia's original arguments will also be clarified.

There are several methods for showing
the ergodicity of a Kleinian group $\Gamma$ with respect to the Patterson measure $\mu$
when $\mu(\Lambda_c(\Gamma))>0$. An intuitive explanation is to rely on the density point theorem 
(see \cite[Theorem 4.4.4]{N}).
Namely, if we replace the reference point of $\mu$ with orbit points tending to the density point of $\Lambda_c(\Gamma)$
conically, then
the measure of $\Lambda_c(\Gamma)$ increases and hence must be of full measure by the invariance under $\Gamma$.
There are various versions of the density point theorem originating from Lebesgue's theorem. In Section \ref{4}, we verify
that the version from Federer \cite{F} is suitable for finite Borel measures on the boundary $\partial X$ of the Gromov hyperbolic space
and a family of shadows as covering subsets.

\medskip
\noindent
{\it Acknowledgments.} Theorems \ref{i-main} and \ref{i-proper} were studied by the first two authors
and announced in the conference
``Rigidity School'' held at University of Tokyo on March 19, 2012. 
Theorem \ref{i-lower} began as a different subject by the third author but was recently incorporated in the present study.

\section{Preliminaries}

In this section, we summarize several properties of Gromov hyperbolic spaces, their discrete isometry groups,
and quasi-invariant quasiconformal measures that are necessary in our arguments of this study.
We present them here by dividing the entire section into subsections.

\subsection{Gromov hyperbolic space and its boundary}\label{2.1}
A geodesic metric space $(X,d)$ is called {\it $\delta$-hyperbolic} for $\delta \geq 0$ if,
for every geodesic triangle $(\alpha, \beta, \gamma)$ in $X$, any edge, say $\alpha$, is
contained in the closed $\delta$-neighborhood of the union $\beta \cup \gamma$ of the other edges.
We call a $\delta$-hyperbolic space $(X,d)$ for some $\delta \geq 0$ a {\it Gromov hyperbolic space}.
Henceforth, we further assume that a $\delta$-hyperbolic space $(X,d)$ is {\it proper} and
has a fixed base point $o \in X$. Concerning the fundamental properties of the Gromov hyperbolic space
mentioned in this subsection, the reader may refer to the lecture note by Coornaert, Delzant, and Papadopoulos \cite{CDP}. 

We consider geodesic rays $\sigma:[0,\infty) \to X$ with
arc length parameter starting from the base point $o$.
Two such geodesic rays $\sigma_1$ and $\sigma_2$ are regarded as asymptotically equivalent if there is some constant $K$ such that
$d(\sigma_1(t),\sigma_2(t)) \leq K$ for all $t \geq 0$. Then, the space of all geodesic rays
based at $o$ modulo the asymptotic equivalence defines a boundary $\partial X$ of $X$, which yields the compactification 
$\overline X=X \cup \partial X$ under the compact-open topology on the space of geodesic rays.
We see that $\overline X$ is a compact Hausdorff space satisfying the second countability axiom. 
Every isometric automorphism of $X$ extends to a self-homeomorphism of $\overline X$.

The characterization of $\delta$-hyperbolicity by triangles also provides the following properties
of geodesics possibly of infinite length (see Ohshika \cite[Proposition 2.61]{O}).

\begin{proposition}\label{geodesic}
For a $\delta$-hyperbolic space $(X,d)$, there is a constant $\kappa(\delta) \geq 0$ depending only on $\delta$
that satisfies the following properties:
\begin{enumerate}
\item
for an ideal geodesic triangle $(\alpha, \beta, \gamma)$ in $X$, some of whose vertices are on $\partial X$,
any edge $\alpha$ is contained in the closed $\kappa(\delta)$-neighborhood of the union $\beta \cup \gamma$ of the other edges;
\item
for any geodesics $\alpha$ and $\beta$ with the same endpoints possibly on $\partial X$, one is contained 
in the closed $\kappa(\delta)$-neighborhood of the other.
\end{enumerate}
\end{proposition}

Hereafter, the constant $\kappa(\delta)$ will always refer to that in the above proposition. 

Another metric can be defined on the compactification $\overline X=X \cup \partial X$.
We choose a so-called visual parameter $a \in (1,a_0(\delta))$ where $a_0(\delta)$ is some constant
depending only on $\delta$. We fix this $a$ and do not move it throughout this paper.
Then, there is a {\it visual metric} $d_a$ on $\overline X$ with respect to the base point $o$ 
that satisfies the following properties. 
\begin{enumerate}
\item
The topology on $\overline X$ induced by the visual metric $d_a$ coincides with
the topology of the compactification of $(X,d)$.
\item
There exists a constant $\lambda=\lambda(\delta,a) \geq 1$ such that
for any geodesic line $(\xi,\eta)$ connecting any $\xi, \eta \in \partial X$, we have
$$
\lambda^{-1}\, a^{-d(o, (\xi,\eta))} \leq d_a(\xi,\eta) \leq \lambda a^{-d(o, (\xi,\eta))}.
$$
\end{enumerate}
This is an analog of the Euclidean metric on $\mathbb B^{n+1} \cup \S1^n$ for the unit ball model $\mathbb B^{n+1}=\{ x \in \R^{n+1} \mid |x|<1\}$ of 
the hyperbolic space $\H2^{n+1}$ and its boundary $\S1^n=\{|x|=1\}$.

\begin{remark}
We often use the following notation.
For $A>0$, $B>0$, and $c \geq 1$, the notation $A \asymp_c B$ implies that $c^{-1}A\leq B\leq cA$.
Thus, the above (2) can be rewritten as $d_a(\xi,\eta) \asymp_{\lambda} a^{-d(o, (\xi,\eta))}$.
\end{remark}

\subsection{Horospherical structure}
For a $\delta$-hyperbolic space $(X,d)$ with base point $o \in X$, 
we can define an analog of a horosphere of $(\mathbb B^{n+1}, d_H)$
as a level set of the Busemann function.
For a given point $\xi \in \partial X$, let $\sigma:[0,\infty) \to X$ be a geodesic ray
such that $\sigma(0)=o$ and $\sigma(\infty)=\lim_{t \to \infty}\sigma(t)=\xi$.
Then, the {\it Busemann function} at $\xi$ is defined as
$$
h_\xi(z)=\lim_{t \to \infty}(d(z,\sigma(t))-d(o,\sigma(t))).
$$
This depends on the choice of the geodesic ray $\sigma$, but the difference is uniformly bounded by
some constant depending only on $\delta$. 

We define the {\it Poisson kernel} by
$k(z,\xi)=a^{-h_\xi(z)}$ for $z \in X$ and $\xi \in \partial X$,
because it plays a similar role to the Poisson kernel $(1-|z|^2)/|z-\xi|^2$ of $\mathbb B^{n+1}$.
Then, the analog of the conformal derivative of an isometric automorphism $\gamma \in \Isom(X,d)$
at $\xi \in \partial X$ is given by
$$
j_\gamma(\xi)=a^{-h_\xi(\gamma^{-1}(o))}=k(\gamma^{-1}(o),\xi).
$$
We remark that $k(z,\xi)$ is defined by the choice of a family of geodesic rays $\sigma$ and
not necessarily a measurable function of $\xi \in \partial X$ in this definition; however, 
this is not a problem in our arguments.

\begin{proposition}\label{poisson}
For any $z \in X$, $\xi \in \partial X$, and $\gamma \in \Isom(X,d)$, the Poisson kernel satisfies
$$
a^{-2\kappa(\delta)} \frac{k(z,\xi)}{j_\gamma(\xi)} \leq k(\gamma(z),\gamma(\xi)) 
\leq a^{2\kappa(\delta)} \frac{k(z,\xi)}{j_\gamma(\xi)}.
$$
\end{proposition}

\begin{proof}
Let $\sigma$ be a geodesic ray from $o$ to $\xi$ and let $\sigma'$ be a geodesic ray from $o$ to $\gamma(\xi)$. Then
\begin{align*}
h_{\gamma(\xi)}(\gamma(z))&=\lim_{t \to \infty}(d(\gamma(z),\sigma'(t))-d(o,\sigma'(t)))\\
&=\lim_{t \to \infty}(d(z,\gamma^{-1} \circ \sigma'(t))-d(\gamma^{-1}(o),\gamma^{-1} \circ \sigma'(t))).
\end{align*}
Here, $\gamma^{-1} \circ \sigma'$ is a geodesic ray from $\gamma^{-1}(o)$ to $\xi$. Because $\sigma$ has the same endpoint $\xi$
as $\gamma^{-1} \circ \sigma'$,
Proposition \ref{geodesic} implies that we can replace $\gamma^{-1} \circ \sigma'$ with $\sigma$ in the formula of 
$h_{\gamma(\xi)}(\gamma(z))$ with additive error $2\kappa(\delta)$. On the contrary, by
\begin{align*}
h_{\xi}(z)&=\lim_{t \to \infty}(d(z,\sigma(t))-d(o,\sigma(t)))\\
&=\lim_{t \to \infty}(d(z,\sigma(t))-d(\gamma^{-1}(o),\sigma(t)))+\lim_{t \to \infty}(d(\gamma^{-1}(o),\sigma(t))-d(o,\sigma(t))),
\end{align*}
we see that
$$
|h_{\gamma(\xi)}(\gamma(z))-h_{\xi}(z)+h_{\xi}(\gamma^{-1}(o))| \leq 2\kappa(\delta),
$$
which implies the required estimate.
\end{proof}

Moreover, $h_\xi(z)$ is approximated by the difference of the distances to $z$ and $o$ from a point $x$
sufficiently close to $\xi$. This can be found in Coornaert \cite[Lemme 2.2]{Co}.

\begin{proposition}\label{approx}
For every $\xi \in \partial X$ and every $z \in X$, there is a neighborhood $U_\xi \subset \overline X$
of $\xi$ such that
$$
|h_\xi(z)-(d(z,x)-d(o,x))| \leq c(\delta)
$$
for every $x \in U_\xi \cap X$, where $c(\delta) \geq 0$ is a constant depending only on $\delta$.
\end{proposition}

Hereafter, the constant $c(\delta)$ will always refer to that in the above proposition. 

\subsection{The Poincar\'e series and the limit set}
For a Gromov hyperbolic space $(X,d)$, we denote the group of all isometric automorphisms of $(X,d)$ by $\Isom(X,d)$.
We say that a subgroup $\Gamma \subset \Isom(X,d)$ is {\it discrete} if it acts on $X$ properly discontinuously.

\begin{definition}
For a discrete subgroup $\Gamma \subset \Isom(X,d)$,
the {\it Poincar\'e series} $P^s_\Gamma(z,x)$ of dimension (or exponent) $s \geq 0$ 
with respect to the visual parameter $a$ is given by
$$
P^s_\Gamma(z,x)=\sum_{\gamma \in \Gamma} a^{-sd(z,\gamma(x))}.
$$
We call $z \in X$ the reference point and $x \in X$ the orbit point.
The convergence or divergence of $P^s_\Gamma(z,x)$ is independent of the choice of $z$ and $x$.
The {\it critical exponent} of $\Gamma$ is defined by
$$
e_a(\Gamma)=\inf \,\{s \geq 0 \mid P^s_\Gamma(z,x)<\infty\}.
$$
\end{definition}

We remark that in contrast with the case of Kleinian groups, the critical exponent is possibly infinite.
In this study, we are only interested in the case where it is finite.
The divergence of the Poincar\'e series at the finite critical exponent is a remarkable property.

\begin{definition}
If the critical exponent $e_a(\Gamma)$ of a discrete group $\Gamma \subset \Isom(X,d)$
is finite and 
the Poincar\'e series $P_\Gamma^s(z,x)$ of dimension $s=e_a(\Gamma)$ diverges,
then $\Gamma$ is said to be of {\it divergence type}.
\end{definition}

For example, every quasiconvex cocompact group $\Gamma$ with $e_a(\Gamma)<\infty$
is of divergence type (see Coornaert \cite[Corollaire 7.3]{Co}).

Here, we present certain properties of the Poincar\'e series that are used later. See \cite[Proposition 2.1]{MY1}
in the case of Kleinian groups. There is no difference in the present case.

\begin{proposition}\label{property}
We have the following properties: 
\begin{enumerate}
\item[$(1)$] $P_{\Gamma}^s(z,x)=P_{\Gamma}^s(x,z)$; 
\item[$(2)$]
$P_{\Gamma}^s(g(z),g(x))=P_{\Gamma}^s(x,z)$ for every $g \in \Isom(X,d)$ with $g \Gamma g^{-1}=\Gamma$.
\end{enumerate}
\end{proposition}

\begin{proof}
Property (1) follows from the equality $d(z,\gamma(x))=d(x,\gamma^{-1}(z))$. Property (2) follows from the equality
$d(g(z),\gamma g(z))=d(z,\tilde \gamma(x))$ for
$\tilde \gamma \in \Gamma$ with $\gamma g=g \tilde \gamma$.
\end{proof}

Next, we define the {\it limit set} $\Lambda(\Gamma)$ of a discrete group $\Gamma$ 
as the set of all accumulation points $\xi$ of
the orbit $\Gamma(x)$ of $x \in X$ in $\partial X$. This is independent of the choice of $x$; hence, 
we may take $x=o$. 
It is known that $\Lambda(\Gamma)$ is a $\Gamma$-invariant closed subset of $\partial X$.
If $\# \Lambda(\Gamma) \geq 3$, then we say that $\Gamma$ is {\it non-elementary}. 

\begin{definition}
For a discrete group $\Gamma \subset \Isom(X,d)$, $\xi \in \Lambda(\Gamma)$ is called a {\it conical limit point}
if there is some geodesic ray $\beta$ toward $\xi$ and a constant $\rho > 0$ such that the orbit $\Gamma(o)$ 
accumulates to $\xi$ within the closed $\rho$-neighborhood of $\beta$. The set of all conical limit points of $\Gamma$ is
called the {\it conical limit set} and denoted by $\Lambda_c(\Gamma)$.
\end{definition}

We utilize the exhaustion of $\Lambda_c(\Gamma)$ by a sequence of $\Gamma$-invariant 
subsets defined by $\rho$. Namely, for a fixed $\rho > 0$,
$\xi \in \Lambda_c(\Gamma)$ belongs to the conical limit subset $\Lambda_c^{(\rho)}(\Gamma)$ if $\Gamma(o)$ 
accumulates to $\xi$ within the closed $\rho$-neighborhood of some geodesic ray $\beta$ toward $\xi$. Then
$\Lambda_c(\Gamma)=\bigcup_{\rho > 0} \Lambda_c^{(\rho)}(\Gamma)$.

\subsection{Quasiconformal measure}
In the Kleinian group case, conformal measures on the boundary at infinity play a central role
in the Patterson--Sullivan theory. On the Gromov hyperbolic space, such measures are allowed to have some ambiguity,
which are called quasiconformal measures. They were introduced by Coornaert \cite{Co}.

We first define this concept for a family of measures labeled by all points in the Gromov hyperbolic space $X$,
and then formulate its quasi-invariance under a subgroup $\Gamma \subset \Isom(X,d)$. 
For Kleinian groups, this way of defining those measures can be 
found in Nicholls \cite{N}.

\begin{definition}
A family $\{\mu_z\}_{z \in X}$ of finite positive Borel measures on $\partial X$ is called a {\it quasiconformal measure family}
of dimension $s \geq 0$ if 
the following conditions are satisfied:
\begin{enumerate}
\item
$\mu_z$ and $\mu_{z'}$ are mutually absolutely continuous for any $z, z' \in X$; 
\item
there is a constant $C \geq 1$ such that the Radon--Nikodym derivative satisfies 
$$
C^{-s}k(z,\xi)^s \leq \frac{d\mu_z}{d\mu_o}(\xi) \leq C^s k(z,\xi)^s \quad (\almostall \xi \in \partial X) 
$$
for every $z \in X$.
\end{enumerate}
We call $C$ the {\it quasiconformal constant}.
\end{definition}

\begin{definition}
For a subgroup $\Gamma \subset \Isom(X,d)$,
an $s$-dimensional quasiconformal measure family $\{\mu_z\}_{z \in X}$ is called {\it $\Gamma$-quasi-invariant} if
there is a constant $D \geq 1$ such that
$$
D^{-1} \leq \frac{d(\gamma^*\mu_{\gamma(z)})}{d\mu_z}(\xi) \leq D \quad (\almostall \xi \in \partial X) 
$$
for every $z \in X$ and for every $\gamma \in \Gamma$. Here, $\gamma^*\mu$ denotes the pull-back of the measure $\mu$ by $\gamma$.
To specify the quasi-invariance constant $D$, we call it $(\Gamma,D)$-quasi-invariant.
\end{definition}

The condition that $\{\mu_z\}_{z \in X}$ is $(\Gamma,D)$-quasi-invariant implies a condition on 
a single positive finite Borel measure $\mu=\mu_o$ at the base point as follows. 

\begin{proposition}\label{singlemeasure}
If an $s$-dimensional quasiconformal measure family $\{\mu_z\}_{z \in X}$ is $(\Gamma,D)$-quasi-invariant
with quasiconformal constant $C \geq 1$, then
$\mu=\mu_o$ satisfies
$$
\widetilde D^{-1} j_\gamma(\xi)^s \leq \frac{d(\gamma^*\mu)}{d\mu}(\xi) \leq \widetilde D j_\gamma(\xi)^s \quad (\almostall \xi \in \partial X)
$$
for every $\gamma \in \Gamma$ with $\widetilde D=C^sD$. 
\end{proposition}

\begin{proof}
The $(\Gamma,D)$-quasi-invariance of $\{\mu_z\}_{z \in X}$ for $z=\gamma^{-1}(o)$ implies
$$
D^{-1} \leq \frac{d(\gamma^*\mu_o)}{d\mu_{\gamma^{-1}(o)}}(\xi) \leq D,  
$$
and the quasiconformality with constant $C$ implies
$$
C^{-s}k(\gamma^{-1}(o),\xi)^s \leq \frac{d\mu_{\gamma^{-1}(o)}}{d\mu_o}(\xi) \leq C^{s}k(\gamma^{-1}(o),\xi)^s.
$$
These two formulae with $k(\gamma^{-1}(o),\xi)=j_\gamma(\xi)$ show the assertion.
\end{proof}

If a positive finite Borel measure $\mu$ on $\partial X$ satisfies the condition in Proposition \ref{singlemeasure},
then we also call it a
$\Gamma$-quasi-invariant or, more precisely, $(\Gamma,\widetilde D)$-{\it quasi-invariant quasiconformal measure} of dimension $s$.

\subsection{The Patterson measure}
Quasi-invariant quasiconformal measures of dimension  at the critical exponent are the main tools in our study.

\begin{definition}
For a non-elementary discrete group $\Gamma \subset \Isom(X,d)$,
a $\Gamma$-quasi-invariant quasiconformal measure $\mu$ or measure family $\{\mu_z\}_{z \in X}$ 
of dimension  at the critical exponent $e_a(\Gamma)<\infty$
with support 
in the limit set $\Lambda(\Gamma)$ is called a
{\it Patterson measure (family)}. 
\end{definition}

\begin{remark}
In this paper, by the support of a measure, we mean the smallest closed set that is of full measure.
As the limit set $\Lambda(\Gamma)$ is the minimal non-empty closed $\Gamma$-invariant subset when $\Gamma$ is non-elementary
(see \cite[Th\'eor\`eme 5.1]{Co}),
if a $\Gamma$-quasi-invariant quasiconformal measure has its support in $\Lambda(\Gamma)$, then it coincides with
$\Lambda(\Gamma)$.
\end{remark}

The existence of a Patterson measure can be verified by
the construction due to Patterson. For a discrete group $\Gamma$ of divergence type,
this is given by a weak-$\ast$ limit of a sequence of 
weighted Dirac masses $m^s_{z,x}$ on $\overline X$ defined by the Poincar\'e series $P_\Gamma^s(z,x)$
as $s$ tends to $e_a(\Gamma)$. This construction naturally produces a $\Gamma$-quasi-invariant 
quasiconformal measure family. We discuss the canonical Patterson measures obtained in this way
in Section \ref{6}.
In addition, the lower bound of the dimensions of
quasi-invariant quasiconformal measures for $\Gamma$ is equal to the critical exponent $e_a(\Gamma)$, 
which is a consequence of
the shadow lemma stated in the following subsection. These results were proved by Coornaert \cite[Th\'eor\`eme 5.4, Corollaire 6.6]{Co} as follows.

\begin{theorem}\label{existence}
Assume that a non-elementary
discrete group $\Gamma \subset \Isom(X,d)$ has
finite critical exponent $e_a(\Gamma)$. Then,
a Patterson measure for $\Gamma$ exists. Moreover,
the dimension $s$ of any $\Gamma$-quasi-invariant quasiconformal measure is at least $e_a(\Gamma)$.
\end{theorem}

We note that if $\Gamma$ is of divergence type, then
every $\Gamma$-quasi-invariant quasiconformal measure $\mu$ of dimension $e_a(\Gamma)$ must have its support in
the limit set $\Lambda(\Gamma)$, which was mentioned in \cite[Lemma 3.7]{MY2}.
This means that $\mu$ is a Patterson measure. 

\begin{proposition}\label{support}
Assume that a 
discrete group $\Gamma \subset \Isom(X,d)$ is of divergence type.
If $\mu$ is a $\Gamma$-quasi-invariant quasiconformal measure of dimension $e_a(\Gamma)$,
then the support of $\mu$ is on the limit set $\Lambda(\Gamma)$.
\end{proposition}

In Section \ref{4}, we show that any Patterson measure $\mu$ has full measure on
the conical limit set $\Lambda_c(\Gamma)$ 
when $\Gamma$ is of divergence type.

The following is an immediate consequence of Proposition \ref{support}.

\begin{corollary}\label{inclusion}
Assume that a non-elementary discrete group $G \subset \Isom(X,d)$ contains a discrete subgroup $\Gamma$ of divergence type.
If $\Lambda(G) \supsetneqq \Lambda(\Gamma)$, then $e_a(G) > e_a(\Gamma)$.
\end{corollary}

\begin{proof}
Suppose that $e_a(G) = e_a(\Gamma)$. Then, the Patterson measure $\mu$ for $G$ whose support coincides with $\Lambda(G)$
is also a $\Gamma$-quasi-invariant quasiconformal measure of dimension $e_a(G) =e_a(\Gamma)$.
By Proposition \ref{support}, the support of $\mu$ is in $\Lambda(\Gamma)$. This implies that $\Lambda(G) =\Lambda(\Gamma)$.
\end{proof}

\subsection{The shadow lemma}
The shadow of a ball in the Gromov hyperbolic space $X$ connects the measure on the boundary $\partial X$ with the geometry of $X$.
The shadow lemma is a fundamental tool in the Patterson--Sullivan theory.

\begin{definition}
Let $B(x,r)$ be the closed ball of center $x \in X$ and radius $r \geq 0$.
For a light source $\omega \in \overline X$, the {\it shadow} is defined by
$$
S_\omega(x,r)=\{\xi \in \partial X \mid \forall [\omega,\xi] \cap B(x,r) \neq \emptyset \},
$$
where $\forall [\omega,\xi)$ refers to every geodesic line or ray connecting $\omega$ and $\xi$. In addition, the extended shadow is given by
$$
\widehat S_\omega(x,r)=\{y \in \overline X \mid \forall [\omega,y) \cap B(x,r) \neq \emptyset \}.
$$
\end{definition}

The shadow lemma is based on the following estimate for the Poisson kernel.
If we take $z=\gamma^{-1}(o)$ for $\gamma \in \Isom(X,d)$, this turns out to be the estimate for $j_\gamma(\xi)$.
This was essentially given in \cite[Lemme 6.1]{Co}.

\begin{lemma}\label{kernel}
There is a constant $C=C(\delta,a) \geq 1$ such that if $o \notin \widehat S_\omega(z,r)$, then
$$
C^{-1}a^{d(o,z)-2r} \leq k(z,\xi) \leq Ca^{d(o,z)}
$$
for every $\xi \in S_\omega(z,r)$.
\end{lemma}

\begin{proof}
For a tree $X$, that is, a $0$-hyperbolic space, the statement can be easily verified. Then, we apply 
approximation by trees as in \cite[Th\'eor\`eme 1.1]{Co}.
\end{proof}

The complement of a shadow can be arbitrarily small if we make the radius sufficiently large.
This geometric observation \cite[Lemme 6.3]{Co} is also used in the proof of the shadow lemma below.
For later purposes, we extend it slightly and provide a proof. Here, $\diam_a$ denotes the diameter 
with respect to the visual metric $d_a$.

\begin{proposition}\label{complement}
For every $\varepsilon>0$, there is a constant $r(\varepsilon)>0$ such that if $r \geq r(\varepsilon)$, then
$$
\diam_a (\partial X-S_\omega(o,r)) \leq \varepsilon
$$
for every $\omega \in \overline X$.
\end{proposition}

\begin{proof}
We may assume that $\omega \notin B(o,r)$.
For any $\xi, \eta \in \partial X-S_\omega(o,r)$, we take some geodesic lines or rays $[\omega,\xi)$ and
$[\omega,\eta)$ that do not intersect $B(o,r)$. For any geodesic line $(\xi,\eta)$,
we consider the geodesic triangle
$\Delta(\omega,\xi,\eta)$ with edges $(\xi,\eta)$, $[\omega,\xi)$, and $[\omega,\eta)$.
As $(\xi,\eta)$ is within a distance $\kappa(\delta)$ from the union $[\omega,\xi) \cup [\omega,\eta)$ by Proposition \ref{geodesic},
we have $d(o,(\xi,\eta)) \geq r-\kappa(\delta)$. By the estimate for the visual metric,
$$
d_a(\xi,\eta) \leq \lambda a^{-d(o,(\xi,\eta))} \leq \lambda a^{\kappa(\delta)-r}
$$
for the constant $\lambda=\lambda(\delta,a) \geq 1$. Hence, we can choose $r(\varepsilon)$ so that 
$\lambda a^{\kappa(\delta)-r(\varepsilon)} \leq \varepsilon$.
\end{proof}

The following theorem was proved in \cite[Proposition 6.1]{Co}.

\begin{theorem}[Shadow lemma]\label{shadowlemma}
Let $\Gamma \subset \Isom(X,d)$ be a non-elementary discrete group and
let $\mu$ be a $\Gamma$-quasi-invariant quasiconformal measure of dimension $s$.
We fix a light source $\omega \in \overline X$.
Then, there are constants $L \geq 1$ and $r_0>0$ such that
$$
L^{-1}a^{-sd(o,\gamma^{-1}(o))} \leq \mu(S_\omega(\gamma^{-1}(o),r)) \leq La^{2rs} a^{-sd(o,\gamma^{-1}(o))} 
$$
for every $\gamma \in \Gamma$ with $o \notin \widehat S_\omega(\gamma^{-1}(o),r)$ and for every $r \geq r_0$.
\end{theorem}

The conical limit set $\Lambda_c(\Gamma)$ can be described by the limit superior of the family of shadows 
$\{S_o(\gamma(o),r)\}_{\gamma \in \Gamma}$. More precisely, by setting $r=\rho+\kappa(\delta)$ for each $\rho > 0$,
Proposition \ref{geodesic} implies that
$$
\limsup_{\gamma \in \Gamma} \{S_o(\gamma^{-1}(o),\rho)\} \subset
\Lambda_c^{(\rho)}(\Gamma) \subset \limsup_{\gamma \in \Gamma} \{S_o(\gamma^{-1}(o),r)\}.
$$
Then, $\Lambda_c(\Gamma)=\bigcup_{\rho > 0}\Lambda_c^{(\rho)}(\Gamma)$ coincides with the limit of the right-hand side (left-hand side) as $r \to \infty$ ($\rho \to \infty$).

By this description of $\Lambda_c(\Gamma)$ and Theorem \ref{shadowlemma}, we have the following claim.
In Section \ref{3}, we show that the converse of this statement is also true.

\begin{proposition}\label{convergence}
Let $\Gamma \subset \Isom(X,d)$ be a non-elementary discrete group and let
$\mu$ be an $s$-dimensional $\Gamma$-quasi-invariant quasiconformal measure on $\partial X$.
If the Poincar\'e series $P_\Gamma^s(z,x)$ converges, then the measure of the conical limit set
$\mu(\Lambda_c(\Gamma))$ is zero.
\end{proposition}

\begin{proof}
We choose the constant $r_0>0$ in Theorem \ref{shadowlemma}
for $\Gamma$ and $\mu$ 
and prove that $\mu(\Lambda_c^{(\rho)}(\Gamma))=0$ for every $\rho \geq r_0-\kappa(\delta)$.
As $P_\Gamma^s(o,o)<\infty$, Theorem \ref{shadowlemma} implies that 
$$
\mu(\bigcup_{\gamma \in \Gamma'}S_o(\gamma^{-1}(o),r)) \leq \sum_{\gamma \in \Gamma'} \mu(S_o(\gamma^{-1}(o),r)) 
\leq La^{2rs} \sum_{\gamma \in \Gamma'} a^{-sd(o,\gamma^{-1}(o))} <\infty
$$
for $r=\rho+\kappa(\delta) \geq r_0$, where $\Gamma'$ is $\Gamma$ minus possibly finitely many elements. 
Then, we see that the measure of
the limit superior of $\{S_o(\gamma^{-1}(o),r)\}$ is zero.
\end{proof}

\begin{corollary}\label{only}
Let $\Gamma \subset \Isom(X,d)$ be a non-elementary discrete group and let
$\mu$ be a $\Gamma$-quasi-invariant quasiconformal measure on $\partial X$.
If $\mu(\Lambda_c(\Gamma))>0$, then $\Gamma$ is of divergence type and $\mu$ is
a Patterson measure for $\Gamma$.
\end{corollary}

\begin{proof}
Let $s$ be the dimension of $\mu$. By Proposition \ref{convergence}, we have $P_\Gamma^s(z,x)=\infty$;
hence, $s \leq e_a(\Gamma)$.
Then Theorem \ref{existence} implies that $s=e_a(\Gamma)$; therefore, $\Gamma$ is of divergence type.
Moreover, by Proposition \ref{support}, $\mu$ should be a Patterson measure.
\end{proof}

We can also claim that $\mu$ has no atoms on $\Lambda_c(\Gamma)$.

\begin{proposition}\label{noatom}
Let $\Gamma \subset \Isom(X,d)$ be a non-elementary discrete group and let
$\mu$ be an $s$-dimensional $\Gamma$-quasi-invariant quasiconformal measure on $\partial X$.
Then, $\mu$ has no point mass on a conical limit point $\xi \in \Lambda_c(\Gamma)$.
\end{proposition}

\begin{proof}
There is some $r >0$ and a sequence $\{\gamma_n\}_{n=1}^\infty \subset \Gamma$ such that
$\gamma_n^{-1}(o)$ converge to $\xi$ as $n \to \infty$ and $\xi \in S_o(\gamma_n^{-1}(o),r)$ for every $n \in \mathbb N$. 
Then, by Proposition \ref{singlemeasure} and Lemma \ref{kernel}, 
there are constants $\widetilde D \geq 1$ and $C \geq 1$ such that
$$
\mu(\{\gamma_n(\xi)\})=(\gamma_n^*\mu)(\{\xi \})
\geq \widetilde D^{-1}(C^{-1}a^{d(o,\gamma_n^{-1}(o))-2r})^s\mu(\{\xi\}).
$$
This implies that if $\mu(\{\xi \})>0$ and $s>0$ then $\mu(\{\gamma_n(\xi)\}) \to \infty$ as $n \to \infty$.
If $\mu(\{\xi \})>0$ and $s=0$ (even though in fact, we have $s>0$ by Proposition \ref{positive} later),
then $\mu(\Gamma(\xi))=\infty$. 
Both cases are impossible; hence, $\mu(\{\xi \})=0$.
\end{proof}

\medskip
\section{Divergence type and measure on the conical limit set}\label{3}

We have seen in Proposition \ref{convergence} that if the $s$-dimensional
Poincar\'e series $P^s_{\Gamma}(z,x)$ converges for a discrete group $\Gamma \subset \Isom(X,d)$, 
then the conical limit set $\Lambda_c(\Gamma)$ has null measure for 
any $s$-dimensional $\Gamma$-quasi-invariant quasiconformal measure $\mu$ on $\partial X$.
In this section, we prove the converse of this statement.

\begin{theorem}\label{main1}
Let $\Gamma \subset \Isom(X,d)$ be a non-elementary discrete group.
If $\mu(\Lambda_c(\Gamma))=0$ for an $s$-dimen\-sional $\Gamma$-quasi-invariant quasiconformal measure $\mu$ on $\partial X$, 
then $P^s_{\Gamma}(z,x)$ converges.
\end{theorem}

For Kleinian groups, this result was proved by Sullivan \cite{S1, S2} by considering ergodicity of the
geodesic flow (see also Roblin \cite[Th\'eor\`eme 1.7]{R1} for a complete argument). Later, Tukia \cite{T1} gave an elementary proof for it.
His arguments are applicable to discrete isometry groups of Gromov hyperbolic spaces if certain modifications are made.
In what follows, we will carry this out respecting Tukia's arguments.

As in Proposition \ref{geodesic},
for a $\delta$-hyperbolic space $(X,d)$, we choose the constant $\kappa(\delta) \geq 0$ such that 
for every geodesic triangle or bi-angle possibly with vertices on the boundary $\partial X$, each edge is
contained in the closed $\kappa(\delta)$-neighborhood of the union of the others.

We utilize shadows to prove Theorem \ref{main1}. In this section, we always put the light source $\omega$ of a shadow 
on the boundary $\partial X$. The following Lemma 
\ref{radiuschange} provides a fundamental technique for considering the inclusion relation between
two shadows. 
As this will also be used later in another case where $\omega$ is in $X$, 
it is generally assumed that $\omega \in \overline X$ only in this lemma.

\begin{lemma}\label{radiuschange}
For any constants $r' \geq r \geq \rho \geq \rho' \geq 0$ with $r-\rho \geq \kappa(\delta)$,  
if 
$$
B(x,\rho') \cap \widehat S_\omega(z,\rho) \neq \emptyset
$$
for any $z,x \in X$ with $d(z,x) >4r'+\kappa(\delta)$ and for any $\omega \in \overline X$, then
$$
B(x,r') \subset \widehat S_\omega(z,r).
$$
\end{lemma}

\begin{proof}
We choose a point $x' \in B(x,\rho') \cap \widehat S_\omega(z,\rho)$ and any geodesic ray $(\omega,x']$
(or geodesic segment $[\omega,x']$) from $x'$ to $\omega$.
Then, $(\omega,x'] \cap B(z,\rho) \neq \emptyset$, so we can take a point $p$ in this intersection.
It should be noted that $d(x,x') \leq \rho'$ and $d(z,p) \leq \rho$.
For any $y \in B(x,r')$, we have that $d(x',y) \leq \rho'+r'$. We consider 
any triangle $\Delta(\omega,x',y)$ with the vertex $y$ and
the edge $(\omega,x']$ containing $p$. 
This edge is contained in the closed $\kappa(\delta)$-neighborhood of the union of the other edges
$(\omega,y] \cup [x',y]$. 
It follows that for every geodesic ray $(\omega,y]$ and for every geodesic segment $[x',y]$,
there is a point $p' \in (\omega,y] \cup [x',y]$ such that $d(p,p') \leq \kappa(\delta)$.
We want to have $p' \in (\omega,y]$. 

To see this, we show that $d(p,[x',y])>\kappa(\delta)$
for every geodesic segment $[x',y]$. As $d(x,x') \leq \rho'$ and $d(x',y) \leq \rho'+r'$,
the distance from $x$ to each point in $[x',y]$ is at most $2\rho'+r' \leq 3r'$.
Using this together with
$d(z,x) > 4r'+\kappa(\delta)$ and $d(z,p) \leq \rho \leq r'$, we obtain $d(p,[x',y])>\kappa(\delta)$. 
Hence, $p' \in (\omega,y]$. 

Furthermore, as $d(z,p) \leq \rho$, $d(p,p') \leq \kappa(\delta)$, and $r-\rho \geq \kappa(\delta)$,
we have $d(z,p') \leq r$, that is, $p' \in B(z,r)$. Hence, $(\omega,y] \cap B(z,r) \neq \emptyset$.
As $y$ is an arbitrary point of $B(x,r')$ and this conclusion is valid for every geodesic ray $(\omega,y]$,
we conclude that $B(x,r') \subset \widehat S_\omega(z,r)$.
\end{proof}

In the next two claims, we consider the influence of slightly moving the light source $\omega \in \partial X$.

\begin{lemma}\label{nbhd@infty}
For $\omega_0 \in \partial X$, let $D \subset X$ be a domain with $\omega_0 \notin \overline D$.
Let $r \geq 0$ be any constant.
Then, there exists a neighborhood
$V \subset \partial X$ of $\omega_0$ such that 
if $\omega \in V$ and $x \in D$, then $\widehat S_\omega(x,r) \subset \widehat S_{\omega_0}(x,r')$ for any $r' \geq r+\kappa(\delta)$.
\end{lemma}

\begin{proof}
We can choose a neighborhood $V$ of $\omega_0$ in $\partial X$ so that the distance
from every point in the closed $r$-neighborhood $N_r(D)$ of $D$ to some geodesic line with
the endpoints $\omega_0$ and any $\omega \in V$ is greater than $\kappa(\delta)$. 
Then, every point $y \in N_r(D)$ on
a geodesic line $(\xi,\omega)$ with endpoints $\xi \in \partial X$ and $\omega \in V$
is within distance $\kappa(\delta)$ of any geodesic line $(\xi,\omega_0)$ with the endpoints $\xi$ and $\omega_0$
by Proposition \ref{geodesic}. 

For any $\omega \in V$ and $x \in D$,
we prove that $S_\omega(x,r)$ is contained in $S_{\omega_0}(x,r')$.
We take an arbitrary $\xi \in S_\omega(x,r)$ and choose some $y \in (\xi,\omega) \cap B(x,r)$. As $y \in N_r(D)$,
every geodesic line $(\xi,\omega_0)$ contains a point $y'$ 
with $d(y,y') \leq \kappa(\delta) \leq r'-r$. Hence, $y' \in B(x,r')$. In particular, 
$(\xi,\omega_0) \cap B(x,r') \neq \emptyset$. This implies that $\xi \in S_{\omega_0}(x,r')$, and thus the inclusion
$S_\omega(x,r) \subset S_{\omega_0}(x,r')$ is proved. The required inclusion
$\widehat S_\omega(x,r) \subset \widehat S_{\omega_0}(x,r')$ then follows from this.
\end{proof}

The neighborhood $V$ of $\omega_0 \in \partial X$
given in Lemma \ref{nbhd@infty} for $D=\widehat S_{\omega_0}(o,r)$
and for a constant $r \geq 0$ is denoted by $V(\omega_0,r)$.

\begin{proposition}\label{form1}
Let $\omega_0 \in \partial X$ and
$r \geq 0$. 
If $B(x,r) \cap \widehat S_{\omega}(z,r) \neq \emptyset$ is satisfied for $\omega \in V(\omega_0,r+\kappa(\delta))$, 
$z \in \widehat S_{\omega_0}(o,r+\kappa(\delta))$, and $x \in X$ with
$d(z,x) >4r+9\kappa(\delta)$, 
then $B(x,r') \subset \widehat S_{\omega_0}(z,r')$ for $r'=r+2\kappa(\delta)$.
\end{proposition}

\begin{proof}
By Lemma \ref{nbhd@infty} applied to $D=\widehat S_{\omega_0}(o,r+\kappa(\delta))$, 
we see that $\widehat S_{\omega}(z,r+\kappa(\delta)) \subset \widehat S_{\omega_0}(z,r+2\kappa(\delta))$ for any
$\omega \in V(\omega_0,r+\kappa(\delta))$ and 
$z \in \widehat S_{\omega_0}(o,r+\kappa(\delta))$. 
On the contrary,
by Lemma \ref{radiuschange}, 
we see that the condition $B(x,r) \cap \widehat S_\omega(z,r) \neq \emptyset$ 
with $d(z,x) >4r+9\kappa(\delta)$ implies
$B(x,r+2\kappa(\delta)) \subset \widehat S_\omega(z,r+\kappa(\delta))$. Hence, $B(x,r') \subset \widehat S_{\omega_0}(z,r')$ follows
for $r'=r+2\kappa(\delta)$.
\end{proof}

The following notations will be used in the proof of Theorem \ref{main1}.
For $\omega \in \partial X$, $z \in X$, and $r>0$, we consider the subset
$$
A_\omega(z,r)=\{x \in \Gamma(o) \mid B(x,r) \subset \widehat S_\omega(z,r)\}
$$
of the orbit $\Gamma(o) \subset X$. For each integer $i \in \N$, we define subsets $A^i_\omega(z,r)$ and 
$\widehat A^i_\omega(z,r)$ of $A_\omega(z,r)$
inductively as follows:
\begin{align*}
\widehat A^1_\omega(z,r)&=A_\omega(z,r);\\
A^1_\omega(z,r)&=\{x^1 \in \widehat A^1_\omega(z,r) \mid B(x^1,r) \not \subset \widehat S_\omega(x,r) 
\ (\forall x \in \widehat A^1_\omega(z,r))\};\\
&\cdots\\
\widehat A^i_\omega(z,r)&=A_\omega(z,r)-\bigsqcup_{j=1}^{i-1} A^{j}_\omega(z,r);\\
A^i_\omega(z,r)&=\{x^i \in \widehat A^i_\omega(z,r) \mid B(x^i,r) \not \subset \widehat S_\omega(x,r) 
\ (\forall x \in \widehat A^i_\omega(z,r))\}.
\end{align*}
This gives a stratification of the orbit by using the inclusion relation of shadows.

As in the stratification by distance, orbit points in each stratum
have disjoint shadows if they are sufficiently apart.

\begin{lemma}\label{disjoint}
Assume that constants $r,\rho \geq 0$ satisfy $r-\rho \geq \kappa(\delta)$.
For any $\omega \in \partial X$ and $z \in X$
and for any $i \in \N$,
if $x, x' \in A^i_\omega(z,r)$ 
satisfy $d(x,x')>4r+\kappa(\delta)$, then 
$$
\widehat S_\omega(x,\rho) \cap \widehat S_\omega(x',\rho)=\emptyset.
$$
\end{lemma}

\begin{proof}
It is assumed toward a contradiction that $\widehat S_\omega(x,\rho) \cap \widehat S_\omega(x',\rho) \neq \emptyset$.
As $d(x,x')>2\rho$ follows from assumption, $B(x,\rho) \cap B(x',\rho)=\emptyset$. Hence, either
$B(x',\rho) \cap \widehat S_\omega(x,\rho) \neq \emptyset$ or $B(x,\rho) \cap \widehat S_\omega(x',\rho) \neq \emptyset$
is satisfied. We assume the former. The other case is treated similarly.
We apply Lemma \ref{radiuschange} for $r'=r>\rho=\rho'$ with $r-\rho \geq \kappa(\delta)$ to obtain
$B(x',r) \subset \widehat S_\omega(x,r)$. However, this violates the condition that $x$ and $x'$ belong to the same
stratum $A^i_\omega(z,r)$.
\end{proof}

This property can be interpreted in terms of the number of orbit points in each stratum having intersecting shadows.
For $r>0$, let $M(r)$ be the number of orbit points $\Gamma(o)$ in the closed ball $B(o,r)$.

\begin{corollary}\label{Mdis}
For constants $r, \rho \geq 0$ with $r-\rho \geq \kappa(\delta)$,
the family of shadows $\{S_\omega(x,\rho)\}$ taken over all $x \in A^i_\omega(z,r)$ are 
$M(4r+\kappa(\delta))$-disjoint, that is,
for each shadow $S_\omega(x,\rho)$, the number of shadows $S_\omega(x',\rho)$ in the family with 
$S_\omega(x,\rho) \cap S_\omega(x',\rho) \neq \emptyset$ is at most $M(4r+\kappa(\delta))$.
\end{corollary}

\begin{proof}
If $S_\omega(x,\rho) \cap S_\omega(x',\rho) \neq \emptyset$, then $d(x,x') \leq 4r+\kappa(\delta)$ by Lemma \ref{disjoint}.
\end{proof}

We note here that if we go through sufficiently many strata, we can gain a definite distance.

\begin{lemma}\label{depth}
For constants $r \geq 0$ and $\ell \geq 0$, let $m \in \N$ be an integer greater than $M(\ell+2r)$. 
Then, every point $x \in A^m_\omega(z,r)$ for any $z \in \Gamma(o)$ and $\omega \in \partial X$ satisfies $d(z,x) > \ell$. 
\end{lemma}

\begin{proof}
Suppose to the contrary that $d(z,x) \leq \ell$.
We take a sequence $z=x_0,x_1,\ldots,x_m=x$ such that $x_i \in A^i_\omega(z,r)$ and
$B(x_i,r) \subset \widehat S(x_{i-1},r)$ for $1 \leq i \leq m$. We show that these points are
all in $B(z,\ell+2r)$. This contradicts the way of choosing $m$.
Clearly $z=x_0$ and $x=x_m$ belong to $B(z,\ell+2r)$. We only have to show that
$d(z,x_i) \leq d(z,x)+2r$ for every $i=1,\ldots,m-1$. 

Suppose that there is some $i$ such that
$d(z,x_i) > d(z,x)+2r$. Take a point $x' \in [z,x] \cap \partial B(x,r)$, which satisfies
$d(z,x')=d(z,x)-r$. 
Similarly,
$d(z,x_i)-r = d(z,B(x_i,r))$.
From these three conditions, we have 
$$
d(z,x') +2r<d(z,B(x_i,r)). 
$$
However, by
considering a geodesic ray $(\omega,x']$,
which intersects both $B(z,r)$ and $B(x_i,r)$, we can derive a contradiction. 
Indeed, taking a point $z' \in (\omega,x'] \cap B(z,r)$, we can apply
the inequality $d(z',x') \geq d(z',B(x_i,r))$ to show that
\begin{align*}
d(z,x') +2r &\geq d(z',x')-d(z,z')+2r\\
&\geq d(z',B(x_i,r))+r \geq d(z,B(x_i,r)).
\end{align*}
This completes the proof.
\end{proof}

Proposition \ref{form1} and Lemma \ref{depth} imply 
the stability in a certain sense of the structure of the strata under small changes of the light source.

\begin{proposition}\label{form2}
Let $r, r' \geq 0$ be constants such that $r'=r+2\kappa(\delta)$ and let 
$m = M(6r+9\kappa(\delta))$. Then,
$$
\widehat A_{\omega}^{im}(o,r) \subset \widehat A_{\omega_0}^{i}(o,r')
$$
for any $\omega_0 \in \partial X$, $\omega \in V(\omega_0,r)$, and $i \in \N$. 
\end{proposition}

\begin{proof}
For every $x \in \widehat A_{\omega}^{im}(o,r)$, we can choose a sequence
$\{x_0,x_1,\ldots,x_i\} \subset A_{\omega}(o,r)$ such that $x_i=x$, $x_0=o$, and 
$x_j \in \widehat A_{\omega}^{m}(x_{j-1},r)$
for every $j=1,2,\ldots,i$.
For $\ell=4r+9\kappa(\delta)$, Lemma \ref{depth} asserts that $d(x_{j-1},x_j) >\ell$ for each
$j=1,2,\ldots,i$. Then, by the condition $x_{j-1} \in \widehat S_{\omega_0}(o,r+\kappa(\delta))$ 
given in Lemma \ref{nbhd@infty},
Proposition \ref{form1} shows that
$B(x_j,r') \subset \widehat S_{\omega_0}(x_{j-1},r')$  
for all such $j$. This implies that $x=x_i$ belongs to $\widehat A_{\omega_0}^{i}(o,r')$.
\end{proof}

As a final step in the preparation, we note that $X$ is 
covered by finitely many extended shadows $\widehat S_{\omega_j}(o,\rho)$
with $\omega_j \in \partial X$ for $j=1,\ldots,k$.
By the compactness of $\overline X$, this is obvious if we know that every point in $\overline X$ is covered by the interior of an
extended shadow $\widehat S_{\omega}(o,\rho)$ for some $\omega \in \partial X$ with a fixed radius $\rho>0$. 
However, this property is slightly different from the property that every closed ball centered at an orbit point is
entirely contained in one of such extended shadows. We fill this gap with the following claim.

\begin{proposition}\label{finiteshadow}
Assume that constants $r, \rho \geq 0$ satisfy $r-\rho \geq \kappa(\delta)$.
If there are finitely many points $\omega_1, \ldots, \omega_k \in \partial X$ such that
$\bigcup_{j=1}^k \widehat S_{\omega_j}(o,\rho) =\overline X$, then for every $x \in \Gamma(o)-B(o,4r+\kappa(\delta))$
there is some $j=1,\ldots, k$ such that 
$B(x,r) \subset \widehat S_{\omega_j}(o,r)$.
\end{proposition}

\begin{proof}
By assumption, every $x=B(x,0) \in \Gamma(o)$
belongs to some $\widehat S_{\omega_j}(o,\rho)$. We apply Lemma \ref{radiuschange}
for $r'=r>\rho>\rho'=0$. This yields that if $d(o,x)>4r+\kappa(\delta)$, then 
$B(x,r) \subset \widehat S_{\omega_j}(o,r)$.
\end{proof}

The proof of Theorem \ref{main1} can now be carried out.

\medskip
\noindent
{\it Proof of Theorem \ref{main1}.}
Let $\mu$ be an $s$-dimensional $\Gamma$-quasi-invariant quasiconformal measure on $\partial X$ with
$\mu(\Lambda_c(\Gamma))=0$. 
We choose $\rho>0$ and extended shadows $\widehat S_{\omega_j}(o,\rho)$
with $\omega_j \in \partial X$ for $j=1,\ldots,k$
such that $\bigcup_{j=1}^k \widehat S_{\omega_j}(o,\rho) =\overline X$.
By Proposition \ref{finiteshadow}, if we set $r \geq \rho+\kappa(\delta)$, then $B(x,r)$ for every
$x \in \Gamma(o)-B(o,4r+\kappa(\delta))$ is contained in some 
$\widehat S_{\omega_j}(o,r)$.

We prove that 
$$
\sum_{x \in A_{\omega}(o,r)} \mu(S_\omega(x,r)) <\infty
$$
for any $\omega \in \partial X$. By the shadow lemma (Theorem \ref{shadowlemma}), this implies that
$$
\sum_{x \in A_{\omega}(o,r)} a^{-s d(o,x)} <\infty.
$$
As $\Gamma(o)$ is the union of $A_{\omega_1}(o,r)$ and $A_{\omega_2}(o,r)$ except for
the finitely many points contained in $B(o,4r+\kappa(\delta))$,
we obtain that $P^s_\Gamma(o,o)=\sum_{x \in \Gamma(o)}a^{-s d(o,x)} <\infty$.

For a given $\omega_0 \in \partial X$,
we divide $A_{\omega_0}(o,r)$ into $\bigsqcup_{i=1}^\infty A^i_{\omega_0}(o,r)$.
We set 
$$
S_i=\bigcup_{x \in A^i_{\omega_0}(o,r)} S_{\omega_0}(x,r)
=\bigcup_{x \in \widehat A^i_{\omega_0}(o,r)} S_{\omega_0}(x,r), 
$$
which decreases as $i \to \infty$.
Then, $\bigcap_{i} S_i$ is contained in $\Lambda_c(\Gamma)$.
As $\mu(\Lambda_c(\Gamma))=0$ by assumption, we see that
$\mu(S_i) \to 0$ as $i \to \infty$. 

\begin{lemma}\label{key}
Let $r>0$ be a sufficiently large constant.
For each $\omega_0 \in \partial X$ 
and for any $\alpha_0>0$, there exists an integer $I=I(\omega_0,\alpha_0) \in \N$ such that
$$
\sum_{x \in A^i_\omega(o,r)}\mu(S_\omega(x,r)) \leq \alpha_0 \mu(S_\omega(o,r))
$$
for every $\omega \in V(\omega_0,r)$ and for every $i \geq I$.
\end{lemma}

\begin{proof}
For an arbitrary $\varepsilon>0$, we consider
$$
\tilde \varepsilon=\varepsilon \, 
\inf\,\{\mu(S_\omega(o,r)) \mid \omega \in V(\omega_0,r)\},
$$
which is positive for a sufficiently large $r>0$. The above arguments for $r'=r+2\kappa(\delta)$
show that there is some $i_0 \in \N$ such that
$$
\mu(\bigcup_{x \in \widehat A^i_{\omega_0}(o,r')} S_{\omega_0}(x,r')) \leq \tilde \varepsilon
$$
for all $i \geq i_0$. By Proposition \ref{form2}, we have 
$$
\widehat A_{\omega}^{im}(o,r) \subset \widehat A_{\omega_0}^{i}(o,r')
$$
for any $\omega \in V(\omega_0,r)$ and $i \in \N$, where $m=M(6r+9\kappa(\delta))$. 
Moreover, Lemma \ref{nbhd@infty} yields that $S_{\omega}(x,r) \subset S_{\omega_0}(x,r')$.
Hence, by setting $I=mi_0$, we have
$$
\mu(\bigcup_{x \in A^i_{\omega}(o,r)} S_{\omega}(x,r)) \leq \varepsilon \mu(S_\omega(o,r))
$$
for every $i \geq I$.

Here, we apply Corollary \ref{Mdis} for $\rho=r-\kappa(\delta)$. Then
\begin{align*}
\sum_{x \in A^i_\omega(o,r)}\mu(S_\omega(x,\rho)) 
&\leq M(4r+\kappa(\delta)) 
\,\mu(\bigcup_{x \in A^i_\omega(o,r)} S_\omega(x,\rho)) \\
&\leq M(4r+\kappa(\delta))
\,\mu(\bigcup_{x \in A^i_\omega(o,r)} S_\omega(x,r)) \leq M(4r+\kappa(\delta)) \varepsilon \mu(S_\omega(o,r)).
\end{align*}
Finally, by the shadow lemma (Theorem \ref{shadowlemma}), if $r$ is sufficiently large, we can find some constant
$\widetilde L \geq 1$ depending on $r-\rho=\kappa(\delta)$
such that $\mu(S_\omega(x,r)) \leq  \widetilde L \mu(S_\omega(x,\rho))$. The conclusion is
$$
\sum_{x \in A^i_\omega(o,r)}\mu(S_\omega(x,r)) \leq \widetilde L M(4r+\kappa(\delta)) \varepsilon \mu(S_\omega(o,r)).
$$
By choosing $\varepsilon>0$ so that $\widetilde L M(4r+\kappa(\delta)) \varepsilon \leq \alpha_0$, we obtain the assertion.
\end{proof}

Hereafter, we choose a sufficiently large $r>0$ that is applicable to the above lemma and fix it.

\begin{proposition}\label{uniform}
For any $\alpha_0>0$, there exists an integer $I_0=I_0(\alpha_0) \in \N$ such that
$$
\sum_{x \in A^i_\omega(o,r)}\mu(S_\omega(x,r)) \leq \alpha_0 \mu(S_\omega(o,r))
$$
for every $\omega \in \partial X$ and every $i \geq I_0$.
\end{proposition}

\begin{proof}
For each $\omega \in \partial X$, we take the neighborhood $V(\omega,r) \subset \partial X$.
As $\partial X$ is compact, we can find finitely many such neighborhoods 
$\{V(\omega_i,r)\}_{i=1}^k$ that cover $\partial X$.
For each $\omega_i$,
we take the integer $I_i=I(\omega_i,\alpha_0)$ as in Lemma \ref{key} and set 
$I_0=\max\{ I_i \mid 1 \leq i \leq k\}$. Then, this satisfies the required property.
\end{proof}

We prove that the uniform estimate in Proposition \ref{uniform} is also valid
even if we replace the base point 
$o$ with an arbitrary orbit point $z \in \Gamma(o)$.

\begin{lemma}\label{orbituniform}
For any $\alpha>0$, there exists an integer $I_*=I_*(\alpha) \in \N$ such that
$$
\sum_{x \in A^i_\omega(z,r)}\mu(S_\omega(x,r)) \leq \alpha \mu(S_\omega(z,r))
$$
for any $z \in \Gamma(o)$ and
$\omega \in \partial X$ with $o \notin \widehat S_\omega(z,r)$
and for every $i \geq I_*$.
\end{lemma}

\begin{proof}
We take any $z \in \Gamma(o)$ and represent it by $z=\gamma^{-1}(o)$ for $\gamma \in \Gamma$.
To show the required estimate, we consider the pull-back $\gamma^* \mu$ of the measure $\mu$.
We note that $\gamma(S_\omega(z,r))=S_{\gamma(\omega)}(o,r)$ and 
$\gamma(A^i_\omega(z,r))=A^i_{\gamma(\omega)}(o,r)$. Hence, from Proposition \ref{uniform},
it follows that
\begin{align*}
\sum_{x \in A^i_\omega(z,r)}(\gamma^*\mu)(S_\omega(x,r)) 
&=\sum_{\gamma(x) \in A^i_{\gamma(\omega)}(o,r)}\mu(S_{\gamma(\omega)}(\gamma(x),r))\\ 
&\leq \alpha_0 \mu(S_{\gamma(\omega)}(o,r))=\alpha_0 (\gamma^*\mu)(S_\omega(z,r))
\end{align*}
for every $i \geq I_0(\alpha_0)$. Thus, it suffices to show that the derivative $(d(\gamma^*\mu)/d\mu)(\xi)$
is in a uniform range on the shadow $S_\omega(z,r)$, which contains $S_\omega(x,r)$
for all $x \in A^i_\omega(z,r)$.

By the $\Gamma$-quasi-invariance of $\mu$, we have 
$$
\frac{d(\gamma^* \mu)}{d\mu}(\xi) \asymp_D j_\gamma(\xi)^s=k(\gamma^{-1}(o),\xi)^s \quad (\almostall \,\xi \in \partial X)
$$
for some constant $D \geq 1$, where $k(z,\xi)$ is the Poisson kernel. On the contrary, by Lemma \ref{kernel},
if $o \notin \widehat S_\omega(z,r)$, then
$$
C^{-1} a^{d(o,z)-2r} \leq k(z,\xi) \leq C a^{d(o,z)} \quad (\xi \in S_\omega(z,r))
$$
for some constant $C \geq 1$ independent of $z=\gamma^{-1}(o)$. Therefore,
$$
\sum_{x \in A^i_\omega(z,r)}\mu(S_\omega(x,r)) \leq \alpha_0 D^2 C^{2s}a^{2sr}\mu(S_\omega(z,r))
$$
for every $i \geq I_0(\alpha_0)$. For $\alpha=\alpha_0 D^2 C^{2s}a^{2sr}$, we just set
$I_*(\alpha)=I_0(\alpha_0)$ to complete the proof.
\end{proof}

\noindent
{\it Proof of Theorem \ref{main1} continued.}
Our goal is to prove that 
$$
\sum_{x \in A_{\omega}(o,r)} \mu(S_\omega(x,r)) <\infty
$$
for any $\omega \in \partial X$. For each $i \in \N$, we set 
$$
Q_i=\sum_{x \in A^i_{\omega}(o,r)} \mu(S_\omega(x,r)). 
$$
For $\alpha=1/2$, we choose the
constant $I_*(1/2) \in \N$ as in Lemma \ref{orbituniform} and define it as $I$.
We can verify that 
$$
\sum_{j=0}^\infty Q_{i+jI} \leq 2 Q_i
$$
for each $i=1,2,\ldots,I$. To see this, we note that 
the condition $o \notin \widehat S_{\omega}(z,r)$ is satisfied for each $z \in A^{i}_{\omega}(o,r)$.
Then, Lemma \ref{orbituniform} implies that 
\begin{align*}
Q_{i+I} =\sum_{x \in A^{i+I}_{\omega}(o,r)} \mu(S_\omega(x,r))
& \leq \sum_{z \in A^{i}_{\omega}(o,r)\ } \sum_{\ x \in A^{I}_{\omega}(z,r)} \mu(S_\omega(x,r))\\
& \leq  \sum_{z \in A^{i}_{\omega}(o,r)} \mu(S_\omega(z,r))/2
= Q_i/2.
\end{align*}
Inductively applying this inequality yields the estimate. Hence,
$$
\sum_{x \in A_{\omega}(o,r)} \mu(S_\omega(x,r))=\sum_{i=1}^I \sum _{j=0}^\infty Q_{i+jI}
\leq 2 \sum_{i=1}^I Q_i.
$$

Finally, we show that each $Q_i$ ($i=1,2,\ldots,I$) is finite. Indeed, Corollary \ref{Mdis}
asserts that for $\rho=r-\kappa(\delta)$, the family $\{S_\omega(x,\rho)\}$ taken over all
$x \in A^i_{\omega}(o,r)$ is $M(4r+\kappa(\delta))$-disjoint. This, in particular, implies that
$$
\sum_{x \in A^i_{\omega}(o,r)} \mu(S_\omega(x,\rho)) \leq M(4r+\kappa(\delta))\mu(S_\omega(o,r)).
$$
As before, the shadow lemma (Theorem \ref{shadowlemma}) yields that $\mu(S_\omega(x,r)) \leq \widetilde L\mu(S_\omega(x,\rho))$
for some constant $\widetilde L \geq 1$. Therefore, we have 
$$
Q_i \leq \widetilde L M(4r+\kappa(\delta))\mu(S_\omega(o,r))<\infty.
$$
This completes the proof of Theorem \ref{main1}.
\qed
\medskip

Theorem \ref{main1} implies that if a non-elementary discrete group
$\Gamma \subset \Isom(X,d)$ is of divergence type, then $\mu(\Lambda_c(\Gamma))>0$
for any $s$-dimen\-sional $\Gamma$-quasi-invariant quasiconformal measure $\mu$ on $\partial X$.
In this situation, $\mu$ has full measure on $\Lambda_c(\Gamma)$.

\begin{corollary}\label{full}
Let $\Gamma \subset \Isom(X,d)$ be a non-elementary discrete group of divergence type and let
$\mu$ be an $s$-dimensional $\Gamma$-quasi-invariant quasiconformal measure on $\partial X$. 
Then, $\mu(\Lambda_c(\Gamma))=\mu(\partial X)$.
\end{corollary}

\begin{proof}
It is assumed toward a contradiction that $\mu(\partial X-\Lambda_c(\Gamma))>0$. Then, the measure $\mu'=\mu|_{\partial X-\Lambda_c(\Gamma)}$
obtained by restricting $\mu$ to $\partial X-\Lambda_c(\Gamma)$
is also an $s$-dimensional $\Gamma$-quasi-invariant quasiconformal measure. Theorem \ref{main1} implies that $\mu'(\Lambda_c(\Gamma))>0$,
but this is a contradiction.
\end{proof}

\medskip
\section{Ergodicity on the conical limit set}\label{4}
In this section, we prove that the action of a discrete group
$\Gamma \subset \Isom(X,d)$ on the conical limit set $\Lambda_c(\Gamma)$ is ergodic 
with respect to any $s$-dimensional $\Gamma$-quasi-invariant quasiconformal measure $\mu$.
We note that this problem is non-trivial only when $\mu(\Lambda_c(\Gamma))>0$.
Hence, we can assume that $\Gamma$ is of divergence type and $\mu$ is a Patterson measure for $\Gamma$ ($s=e_a(\Gamma)$) by 
Corollary \ref{only}.

\begin{theorem}\label{main2}
Let $\Gamma \subset \Isom(X,d)$ be a non-elementary discrete group and let
$\mu$ be an $s$-dimensional $\Gamma$-quasi-invariant quasiconformal measure with full measure on $\mu(\Lambda_c(\Gamma))$.
If a measurable subset $E \subset \Lambda_c(\Gamma)$ in the conical limit set is
$\Gamma$-invariant (for almost every $ \mu$) and $\mu(E)>0$, then $\mu(E)=\mu(\Lambda_c(\Gamma))$.
\end{theorem}

For Kleinian groups, one way to prove the corresponding result is to utilize the density point theorem 
(see Nicholls \cite[Theorem 4.4.4]{N} for example). In the other way, Roblin \cite[pp.22--23]{R1} proved the result
more generally
for discrete isometry groups on ${\rm CAT}(-1)$ spaces. 
His arguments are almost acceptable even in the case of discrete isometry groups of
Gromov hyperbolic spaces; only few modifications such as those in Section \ref{3} are required. Nevertheless, our proof here is
again based on the density point theorem; our purpose is to show that the family of shadows can be adapted to elements
of the density point theorem for Borel measures on metric spaces in general. Concerning this theorem, certain
necessary concepts are introduced in the following from Federer \cite{F}.

\begin{definition}
Let $(\Lambda,d)$ be a metric space and let $\mu$ be a Borel measure on $\Lambda$ for which
every bounded measurable subset has finite measure. A {\it covering relation} $\mathcal C$ is
a subset of the set of all such pairs $\{(\xi,S)\}$ that $S$ is a measurable subset of $\Lambda$ and $\xi$ is a point in $S$.
We say that $\mathcal C$ is {\it fine} at $\xi \in \Lambda$ if 
$$
\inf\,\{\diam(S) \mid (\xi,S) \in \mathcal C\}=0.
$$
For any measurable subset $E \subset \Lambda$, we define a family of subsets of $\Lambda$ by
$$
{\mathcal C}(E)=\{S \subset \Lambda \mid (\xi,S) \in {\mathcal C} \quad (\exists\, \xi \in E)\}.
$$
\end{definition}

\begin{definition}
A covering relation $\mathcal V$ is called a {\it Vitali relation} for a Borel measure $\mu$ on $\Lambda$
if $\mathcal V$ is fine at every $x \in \Lambda$ and if the following condition holds: if $\mathcal C \subset \mathcal V$
is fine at every point $\xi$ of a measurable subset $E \subset \Lambda$, then 
${\mathcal C}(E)$ has a countable disjoint subfamily $\{S_n\}_{n=1}^\infty$ such that $\mu(E - \bigsqcup_{n=1}^\infty S_n)=0$.
\end{definition}

A general density point theorem can be stated as follows (\cite[Theorem 2.9.11]{F}).

\begin{theorem}\label{density}
Let $\mathcal V$ be a Vitali relation for a measure $\mu$ on $\Lambda$ and 
let $E \subset \Lambda$ be a measurable subset. Then, for almost every point $\xi \in E$ with respect to $\mu$,
one has
$$
\lim_{n \to \infty} \frac{\mu(E \cap S_n)}{\mu(S_n)}=1
$$
for every sequence $\{S_n\}_{n=1}^\infty$ such that $(\xi,S_n) \in \mathcal V$ for all $n$ and $\diam S_n \to 0$ as
$n \to \infty$.
\end{theorem}

As a sufficient condition for a Vitali relation, we have the following (\cite[Theorem 2.8.17]{F}).

\begin{lemma}\label{federer}
Let $\mathcal V=\{(\xi,S)\}$ be a covering relation for a measure $\mu$ on $\Lambda$
such that every $S \in {\mathcal V}(\Lambda)$ is a bounded closed subset and $\mathcal V$ is fine at every $\xi \in \Lambda$.
For a non-negative function $f$ on ${\mathcal V}(\Lambda)$ and a constant $\tau \in (1,\infty)$, 
let
$$
\widetilde S=\bigcup \{S' \in {\mathcal V}(\Lambda) \mid S' \cap S \neq \emptyset,\ f(S') \leq \tau f(S)\} \subset \Lambda
$$
for each $S \in {\mathcal V}(\Lambda)$. Suppose that for almost every $\xi \in \Lambda$ with respect to $\mu$
$$
\limsup_{S \to \xi} \left\{f(S)+\frac{\mu(\widetilde S)}{\mu(S)}\right\}
$$
is finite, where the limit superior is taken over all sequences $\{S\}$ with $(\xi,S) \in \mathcal V$ and $\diam S \to 0$.
Then, $\mathcal V$ is a Vitali relation for $\mu$.
\end{lemma}

We apply these results to our case. For a discrete group $\Gamma \subset \Isom(X,d)$ of divergence type, 
we adopt the conical limit subset $\Lambda_c^{(\rho)}(\Gamma)$ for a sufficiently large $\rho>0$
with the restriction of the visual metric $d_a$
as the metric space $(\Lambda,d)$. Moreover, we define $\mu$ to be
the restriction of an $s$-dimensional $\Gamma$-quasi-invariant quasiconformal measure to $\Lambda_c^{(\rho)}(\Gamma)$,
and $\mathcal V$ to be
$$
\mathcal V^{(\rho,r)}
=\{(\xi,S_o^{(\rho)}(x,r)) \mid x \in \Gamma(o),\ \xi \in S_o^{(\rho)}(x,r)\} 
$$
for a fixed $r \geq \rho+\kappa(\delta)$,
where $S_o^{(\rho)}(x,r)= S_o(x,r) \cap \Lambda_c^{(\rho)}(\Gamma)$. 
We note that $S_o^{(\rho)}(x,r) \neq \emptyset$ for each $x \in \Gamma(o)$.
We also consider 
$$
f(S_o^{(\rho)}(x,r))=a^{-d(o,x)}
$$
as the non-negative function $f$.

We show that the covering relation $\mathcal V^{(\rho,r)}$ is a Vitali relation for $\mu$ when $r \geq \rho+\kappa(\delta)$ is sufficiently large.
First, we see that
$\mathcal V^{(\rho,r)}$ is fine at every $\xi \in \Lambda_c^{(\rho)}(\Gamma)$ for each $r \geq \rho+\kappa(\delta)$ from the following proposition.

\begin{proposition}\label{diameter}
The diameter of the shadow satisfies
$$
\diam_a S_o(x,r) \leq \lambda a^{2r} \cdot a^{-d(o,x)}
$$
for any $x \in X$ and $r>0$, where $\lambda=\lambda(\delta,a) \geq 1$ is the constant given in Subsection \ref{2.1}.
\end{proposition}

\begin{proof}
Taking any two points $\xi,\eta$ in $S_o(x,r)$ and some geodesic line $(\xi,\eta)$,
we choose sequences $\{\xi_n\}$ and $\{\eta_n\}$ on $(\xi,\eta)$ such that $\xi_n \to \xi$ and
$\eta_n \to \eta$ as $n \to \infty$. Then, the geodesic segments $[\xi_n,\eta_n] \subset (\xi,\eta)$ clearly satisfy
$d(o,(\xi,\eta))=d(o,[\xi_n,\eta_n])$ for all sufficiently large $n$.
We consider geodesic segments $[o,\xi_n]$ and $[o,\eta_n]$. Passing to subsequences if necessary, we may assume that
$[o,\xi_n]$ converge to a geodesic ray $[o,\xi)$ and $[o,\eta_n]$ converge to a geodesic ray $[o,\eta)$ as $n \to \infty$.
For an arbitrary $\varepsilon >0$, we can find $\xi'_n \in [o,\xi)$ and $\eta'_n \in [o,\eta)$ such that
$d(\xi_n,\xi'_n) \leq \varepsilon$ and $d(\eta_n,\eta'_n) \leq \varepsilon$ for some sufficiently large $n$.
Hereafter, we fix this $n$.

The distance $d(o,(\xi,\eta))=d(o,[\xi_n,\eta_n])$ is bounded from below by the Gromov product
$$
(\xi_n \mid \eta_n)_o:=\frac{1}{2}(d(o,\xi_n)+d(o,\eta_n)-d(\xi_n,\eta_n)) \geq (\xi'_n \mid \eta'_n)_o-2\varepsilon.
$$ 
As $[o,\xi)$ and $[o,\eta)$ intersect
$B(x,r)$, the triangle inequality yields that
$$
d(o,\xi'_n) \geq d(o,x)+d(x,\xi'_n)-2r; \quad d(o,\eta'_n) \geq d(o,x)+d(x,\eta'_n)-2r.
$$
Hence,
$$
d(o,(\xi,\eta)) \geq (\xi_n \mid \eta_n)_o \geq d(o,x)+(\xi'_n \mid \eta'_n)_x-2r-2\varepsilon \geq d(o,x)-2r-2\varepsilon.
$$
As $\varepsilon$ is arbitrary, we conclude that $d(o,(\xi,\eta)) \geq d(o,x)-2r$. Then, the distance on $\partial X$ is estimated as
$$
d_a(\xi,\eta) \leq \lambda a^{-d(o,(\xi,\eta))} \leq \lambda a^{2r} \cdot a^{-d(o,x)}.
$$
Thus, $\diam_a S_o(x,r)$ is bounded by this value.
\end{proof}

Now we are ready to accomplish our purpose.

\begin{lemma}\label{vitali}
Let $\Gamma \subset \Isom(X,d)$ be a non-elementary discrete group and
let $\mu$ be an $s$-dimensional $\Gamma$-quasi-invariant quasiconformal measure. Then, 
$\mathcal V^{(\rho,r)}$ is a Vitali relation for $\mu$ if $r \geq \max\{\rho+\kappa(\delta),r_0\}$, where
$r_0$ is the constant that arises from the shadow lemma.
\end{lemma}

\begin{proof}
Let $\tau=a^r>1$. Concerning the function $f(S_o^{(\rho)}(x,r))=a^{-d(o,x)}$ for each $x \in \Gamma(o)$, 
we see that the condition $f(S_o^{(\rho)}(x',r)) \leq \tau f(S_o^{(\rho)}(x,r))$
is equivalent to that $d(o,x') \geq d(o,x)-r$. Consider $x' \in \Gamma(o)$ that holds this condition.
Let $\tilde r=r+\kappa(\delta)$. Lemma \ref{radiuschange} implies that
if $d(x,x')>5 \tilde r$ and $B(x',r) \cap \widehat S_o(x,r) \neq \emptyset$, then 
$B(x',\tilde r) \subset \widehat S_o(x,\tilde r)$. Here, we see that the latter assumption can be replaced by the condition
$S_o(x',r) \cap S_o(x,r) \neq \emptyset$. Indeed, $B(x,r)$ and $B(x',r)$ are disjoint in this case and the condition
$S_o(x',r) \cap S_o(x,r) \neq \emptyset$ is equivalent to that either $B(x',r) \cap \widehat S_o(x,r) \neq \emptyset$ or
$B(x,r) \cap \widehat S_o(x',r) \neq \emptyset$. However, the assumption $d(x',o) \geq d(x,o)-r$ rules out the latter case, and thus
we have $B(x',r) \cap \widehat S_o(x,r) \neq \emptyset$. 

Now we show that $S_o(x',r) \subset S_o(x,6 \tilde r)$ under the condition $S_o(x',r) \cap S_o(x,r) \neq \emptyset$.
If $d(x,x')>5 \tilde r$, then the above argument concludes that $B(x',r) \subset \widehat S_o(x,\tilde r)$.
This implies, in particular, that $S_o(x',r) \subset S_o(x,\tilde r) \subset S_o(x,6\tilde r)$. Furthermore, if
$d(x,x') \leq 5 \tilde r$, then $B(x',r) \subset B(x,6\tilde r)$, which also implies that $S_o(x',r) \subset S_o(x,6\tilde r)$.

To prove that $\mathcal V^{(\rho, r)}$ is a Vitali relation for $\mu$, we rely on Lemma \ref{federer}. As $\mathcal V^{(\rho, r)}$ is
fine at every $\xi \in \Lambda_c^{(\rho)}(\Gamma)$ by Proposition \ref{diameter}, it suffices to show that
$f(S)+\mu(\widetilde S)/\mu(S)$ is uniformly bounded
for every $S=S^{(\rho)}_o(x,r)$. Clearly $f(S) \leq 1$. On the contrary, $\widetilde S$ is contained in 
$S_o^{(\rho)}(x,6\tilde r)$ as we have seen above. 
Then, the shadow lemma for the $s$-dimensional $\Gamma$-quasi-invariant quasiconformal
measure $\mu$ restricted to 
$\Lambda_c^{(\rho)}(\Gamma)$ gives
\begin{align*}
\mu(\widetilde S) &\leq \mu(S_o^{(\rho)}(x,6\tilde r)) \leq L a^{12 s \tilde r} \cdot a^{-s d(o,x)};\\
\mu(S) &=\mu(S^{(\rho)}_o(x,r)) \geq L^{-1}a^{-s d(o,x)},
\end{align*}
where $L \geq 1$ is a constant independent of $x \in \Gamma(o)$. This implies that
$\mu(\widetilde S)/\mu(S) \leq L^2 a^{12 s \tilde r}<\infty$.
\end{proof}
\medskip

\noindent
{\it Proof of Theorem \ref{main2}.}
We prove that $\mu(E \cap \Lambda_c^{(\rho)}(\Gamma))=\mu(\Lambda_c^{(\rho)}(\Gamma))$ for all sufficiently large $\rho>0$.
Then, because $\Lambda_c(\Gamma)=\bigcup_{\rho>0} \Lambda_c^{(\rho)}(\Gamma)$, we have that 
$\mu(E)=\mu(E \cap \Lambda_c(\Gamma))=\mu(\Lambda_c(\Gamma))$, which is the required result.
Fixing a sufficiently large $\rho$ with $\mu(E \cap \Lambda_c^{(\rho)}(\Gamma))>0$,
we regard $\mu$ as its restriction to $\Lambda_c^{(\rho)}(\Gamma)$ and give a proof for the above fact.

It is assumed toward a contradiction that $\mu(E \cap \Lambda_c^{(\rho)}(\Gamma))<\mu(\Lambda_c^{(\rho)}(\Gamma))$.
By Lemma \ref{vitali}, ${\mathcal V}^{(\rho, r)}$ is a Vitali relation for $\mu$ if $r \geq \max\{\rho+\kappa(\delta),r_0\}$.
Theorem \ref{density}, in particular, asserts that there is a density point $\xi$ of $\Lambda_c^{(\rho)}(\Gamma)-E$ such that
$$
\lim_{n \to \infty} \frac{\mu(S_o^{(\rho)}(\gamma_n^{-1}(o),r)-E)}{\mu(S_o^{(\rho)}(\gamma_n^{-1}(o),r))}=1,\ {\rm or}\
\lim_{n \to \infty} \frac{\mu(E \cap S_o^{(\rho)}(\gamma_n^{-1}(o),r))}{\mu(S_o^{(\rho)}(\gamma_n^{-1}(o),r))}=0,
$$
where $\{\gamma_n\}_{n=1}^\infty \subset \Gamma$ is a sequence such that $\gamma_n^{-1}(o)$ converge to $\xi$ within distance $\rho$ from
some geodesic ray toward $\xi$. We note that the above limit at $\xi$ exists for a fixed $r$. 
However, as there are such density points $\xi$
in full measure for each $r$, 
we can choose a common density point $\xi$ where the limit exists 
for countably many integers $r \geq \max\{\rho+\kappa(\delta),r_0\}$.
By passing to a subsequence, we may assume that $\gamma_n(o)$ converge to some $\eta \in \partial X$.
By Proposition \ref{noatom}, we see that $\mu(\{\eta\})=0$.

We take an arbitrary $\widetilde \varepsilon>0$ such that $\mu(\Lambda_c^{(\rho)}(\Gamma)) \geq 2\widetilde \varepsilon$.
By the regularity of the finite Borel measure 
$\mu$, there is an open ball 
$D(\eta,\varepsilon) \subset \partial X$ centered at $\eta$ with radius $\varepsilon>0$
such that $\mu(D(\eta,\varepsilon)) \leq \widetilde \varepsilon$.
Then, by Proposition \ref{complement}, there is $r(\varepsilon)>0$ such that 
$$
\diam_a(\partial X-S_{\gamma(o)}(o,r)) \leq \varepsilon
$$
for every $\gamma \in \Gamma$ and every $r \geq r(\varepsilon)$.
Fixing such an $r \geq r(\varepsilon)$, we see that 
$\gamma_n(S_o(\gamma_n^{-1}(o),r))=S_{\gamma_n(o)}(o,r)$ does not contain $\eta$
for all sufficiently large $n$; hence,
$$
\mu(\Lambda_c^{(\rho)}(\Gamma)-\gamma_n(S_o^{(\rho)}(\gamma_n^{-1}(o),r))) \leq \widetilde \varepsilon.
$$
By the choice of $\widetilde \varepsilon$, this implies that 
$\mu(\gamma_n(S_o^{(\rho)}(\gamma_n^{-1}(o),r))) \geq \widetilde \varepsilon$.

Now we fix some $r \geq \max\{\rho+\kappa(\delta),r_0,r(\varepsilon)\}$ and apply the above result. By Lemma \ref{kernel},
there is a constant $C \geq 1$ independent of $\gamma \in \Gamma$ such that
$$
C^{-1} a^{d(o,\gamma^{-1}(o))-2r} \leq j_\gamma(\xi)=k(\gamma^{-1}(o),\xi) \leq C a^{d(o,\gamma^{-1}(o))}
$$
for every $\xi \in S_o(\gamma^{-1}(o),r)$. Then
\begin{align*}
\mu(E \cap \gamma_n(S_o^{(\rho)}(\gamma_n^{-1}(o),r)))&=(\gamma_n^*\mu)(E \cap S_o^{(\rho)}(\gamma_n^{-1}(o),r))\\
&\leq D(C a^{d(o,\gamma_n^{-1}(o))})^s\mu(E \cap S_o^{(\rho)}(\gamma_n^{-1}(o),r));\\
\mu(\gamma_n(S_o^{(\rho)}(\gamma_n^{-1}(o),r)))&=(\gamma_n^*\mu)(S_o^{(\rho)}(\gamma_n^{-1}(o),r))\\
&\geq D^{-1}(C^{-1}a^{d(o,\gamma_n^{-1}(o))-2r})^s\mu(S_o^{(\rho)}(\gamma_n^{-1}(o),r)),
\end{align*}
where $D \geq 1$ is the constant for $\Gamma$-quasi-invariance of $\mu$. From these estimates, we have
$$
\frac{\mu(E \cap \gamma_n(S_o^{(\rho)}(\gamma_n^{-1}(o),r)))}{\mu(\gamma_n(S_o^{(\rho)}(\gamma_n^{-1}(o),r)))}\leq
D^2(Ca^r)^{2s}\frac{\mu(E \cap S_o^{(\rho)}(\gamma_n^{-1}(o),r))}{\mu(S_o^{(\rho)}(\gamma_n^{-1}(o),r))},
$$
which tend to $0$ as $n \to \infty$.

As $\mu(\gamma_n(S_o^{(\rho)}(\gamma_n^{-1}(o),r))) \geq \widetilde \varepsilon$
for all sufficiently large $n$, we see that
$$
\mu(E \cap \gamma_n(S_o^{(\rho)}(\gamma_n^{-1}(o),r))) \to 0 \quad (n \to \infty).
$$
Combined with $\mu(\Lambda_c^{(\rho)}(\Gamma)-\gamma_n(S_o^{(\rho)}(\gamma_n^{-1}(o),r))) \leq \widetilde \varepsilon$,
this implies that $\mu(E) \leq \widetilde \varepsilon$.
As we have this conclusion for any sufficiently small $\widetilde \varepsilon>0$, we obtain 
$\mu(E) =0$. However, this contradicts the assumption $\mu(E)>0$, and the proof is complete.
\qed

\begin{corollary}\label{ergodic}
Let $\Gamma \subset \Isom(X,d)$ be a non-elementary discrete group of divergence type and let
$\mu$ be a Patterson measure for $\Gamma$. Then, $\Gamma$ acts on $\partial X$ ergodically
with respect to $\mu$.
\end{corollary}

\begin{proof}
This follows from Corollary \ref{full} and Theorem \ref{main2}.
\end{proof}

\medskip
\section{Quasi-uniqueness of Patterson measures}\label{5}

In this section, we prove that under the assumption of ergodicity of a discrete group
$\Gamma \subset \Isom(X,d)$
with respect to an $s$-dimensional $\Gamma$-quasi-invariant
quasiconformal measure $\mu$, 
any such measure that is absolutely continuous with respect to $\mu$ is unique in a certain sense.
We apply this ``quasi-uniqueness'' to Patterson measures for $\Gamma$ of divergence type.

For later purposes, 
the ambiguity of the uniqueness will be described in terms of the quasi-invariance constants and the total mass
of the measures.
For a measure $\mu$ in general, we denote its total mass by $\Vert \mu \Vert$.

\begin{lemma}\label{useful}
If a discrete group $\Gamma \subset \Isom(X,d)$ acts ergodically on $\partial X$
with respect to a $(\Gamma,D)$-quasiconformal measure $\mu$, then any 
$(\Gamma,D')$-quasiconformal measure $\nu$ that is absolutely continuous with respect to $\mu$ satisfies
$$
(DD')^{-1} \frac{\Vert \nu \Vert}{\Vert \mu \Vert}\leq \frac{d\nu}{d\mu}(\xi) \leq DD'\frac{\Vert \nu \Vert}{\Vert \mu \Vert}
\quad (\almostall \xi \in \partial X).
$$
In particular, $\mu$ is absolutely continuous with respect to $\nu$.
\end{lemma}

\begin{proof}
For simplicity, we may assume that $\Vert \mu \Vert=\Vert \nu \Vert=1$. Let
$$
E=\bigcap_{\gamma \in \Gamma}\,\{\xi \in \partial X \mid (DD')^{-1} \leq \frac{d\nu}{d\mu}(\gamma(\xi)) \leq DD'\},
$$
which is a $\Gamma$-invariant measurable subset of $\partial X$. By ergodicity, we have $\mu(E)=0$ or $\mu(E)=1$.
We prove that $\mu(E)=1$, which shows, in particular, that 
$$
(DD')^{-1} \leq \frac{d\nu}{d\mu}(\xi) \leq DD' \quad (\almostall \xi \in \partial X)
$$
by taking $\gamma=\id$.

It is assumed toward a contradiction that $\mu(E)=0$, that is, $\mu(E^c)=1$ for the complement $E^c$ of $E$.
We divide $E^c$ into two disjoint $\Gamma$-invariant measurable subsets:
\begin{align*}
E^c_+&=\bigcup_{\gamma \in \Gamma}\,\{\xi \in \partial X \mid \frac{d\nu}{d\mu}(\gamma(\xi)) > DD'\};\\
E^c_-&=\bigcup_{\gamma \in \Gamma}\,\{\xi \in \partial X \mid \frac{d\nu}{d\mu}(\gamma(\xi)) < (DD')^{-1}\}.
\end{align*}
Again by ergodicity, we have $\mu(E^c_+)=1$ or otherwise $\mu(E^c_-)=1$. 
For each $n \in \N$, we define
\begin{align*}
(E^c_+)_n&=\bigcup_{\gamma \in \Gamma}\,\{\xi \in \partial X \mid \frac{d\nu}{d\mu}(\gamma(\xi)) > DD'+\frac{1}{n}\};\\
(E^c_-)_n&=\bigcup_{\gamma \in \Gamma}\,\{\xi \in \partial X \mid \frac{d\nu}{d\mu}(\gamma(\xi)) < (DD')^{-1}-\frac{1}{n}\},
\end{align*}
which are also $\Gamma$-invariant. Then, each $\{(E^c_\pm)_n\}_{n=1}^\infty$ is an increasing sequence converging to
$E^c_\pm=\bigcup_{n=1}^\infty (E^c_\pm)_n$.
As 
$\mu((E^c_\pm)_n)$ is either $0$ or $1$ for every $n$, 
there is some $n_0 \in\N$ such that either
$\mu((E^c_+)_{n_0})=1$ or $\mu((E^c_-)_{n_0})=1$. Finally, we consider $\Gamma$-invariant measurable subsets
\begin{align*}
F_+&=\bigcap_{\gamma \in \Gamma}\,\{\xi \in \partial X \mid \frac{d\nu}{d\mu}(\gamma(\xi)) >1+\frac{1}{n_0DD'}\};\\
F_-&=\bigcap_{\gamma \in \Gamma}\,\{\xi \in \partial X \mid \frac{d\nu}{d\mu}(\gamma(\xi)) <1-\frac{DD'}{n_0}\}.
\end{align*}

We see that $F_\pm$ contains $(E^c_\pm)_{n_0}$; hence, $\mu(F_+)=1$ or otherwise $\mu(F_-^c)=0$ is satisfied. 
Indeed, for almost every $\xi \in (E^c_-)_{n_0}$ there is 
some $\gamma_0 \in \Gamma$ such that
$(d\nu/d\mu)(\gamma_0(\xi)) < (DD')^{-1}-1/n_0$. Then, for every $\gamma \in \Gamma$
\begin{align*}
\frac{d\nu}{d\mu}(\gamma(\xi))&=\frac{d(\gamma \gamma_0^{-1})^*\nu}{d(\gamma \gamma_0^{-1})^*\mu}(\gamma_0(\xi))\\
&\leq \frac{D'}{D^{-1}}\cdot\frac{d\nu}{d\mu}(\gamma_0(\xi))<1-\frac{DD'}{n_0},
\end{align*}
which shows that $\xi \in F_-$. The other case for $F_+$ is treated similarly.
Then, $\nu(F_-)=1$ because $\nu(F_-)=1-\nu(F_-^c)$ and $\nu$ is absolutely continuous with respect to $\mu$.
However, we have
$$
\nu(F_-)=\int_{F_-}\frac{d\nu}{d\mu}(\xi) d\mu(\xi) \leq \left(1-\frac{DD'}{n_0}\right)\mu(F_-)<1,
$$
which is a contradiction. We also obtain $\nu(F_+)>1$ in the other case where $\mu(F_+)=1$, 
which leads to a contradiction.
\end{proof}

The quasi-uniqueness is mainly applied to Patterson measures for $\Gamma$ of divergence type.

\begin{theorem}\label{uniqueness}
Let $\Gamma \subset \Isom(X,d)$ be a non-elementary discrete group of divergence type.
Then, any two Patterson measures for $\Gamma$ are mutually absolutely continuous.
If $\mu$ and $\mu'$ are Patterson measures for $\Gamma$ with quasi-invariance constants $D$ and $D'$ respectively, then
$$
(DD')^{-1} \frac{\Vert \mu' \Vert}{\Vert \mu \Vert}\leq \frac{d\mu'}{d\mu}(\xi) \leq DD' \frac{\Vert \mu' \Vert}{\Vert \mu \Vert}
\quad (\almostall \xi \in \partial X).
$$
\end{theorem}

\begin{proof}
If $\mu$ and $\mu'$ are Patterson measures for $\Gamma$, then $\mu+\mu'$ is also a Patterson measure for $\Gamma$.
By Corollary \ref{ergodic}, $\Gamma$ acts ergodically on $\partial X$ with respect to $\mu+\mu'$. 
As $\mu$ and $\mu'$ are absolutely continuous with respect to $\mu+\mu'$, Lemma \ref{useful} implies that
$\mu$ and $\mu'$ are mutually absolutely continuous via $\mu+\mu'$.
Then, the required inequality also follows from Lemma \ref{useful}.
\end{proof}

\medskip
\section{Quasi-invariance under the normalizer}\label{6}
In the previous section, we have seen the ``quasi-uniqueness'' of Patterson measures for 
a divergence-type group $\Gamma \subset \Isom(X,d)$.
Using this property, we show that the Patterson measure is also quasi-invariant under the normalizer of $\Gamma$.
The invariance under the normalizer is a property of the Poincar\'e series and we use the inheritance of this property
to the quasi-unique Patterson measure. To this end, we have to fix the canonical construction of a Patterson measure family
from the weighted Dirac masses defined by the Poincar\'e series, which is the so-called Patterson construction.
See Nicholls \cite{N} in the case of Kleinian groups.

We always assume that a discrete group $\Gamma \subset \Isom(X,d)$ is non-elementary and of divergence type.
For any reference point $z \in X$, orbit point $x \in X$, and exponent $s>e_a(\Gamma)$,
we define a measure on $X$ by
$$
m_{z,x}^s=\frac{1}{P_\Gamma^s(o,x)}\sum_{\gamma \in \Gamma} a^{-sd(z,\gamma(x))} D_{\gamma(x)},
$$
where $D_x$ is the Dirac measure at $x \in X$. In fact, $m_{z,x}^s=m_{z,x'}^s$ if $x' \in \Gamma(x)$.
We note that the total mass $\Vert m_{z,x}^s \Vert$ satisfies
$$
a^{-sd(o,z)} \leq \Vert m_{z,x}^s \Vert=\frac{P_\Gamma^s(z,x)}{P_\Gamma^s(o,x)} \leq a^{sd(o,z)};\quad \Vert m_{o,x}^s \Vert=1.
$$
The measure $m_{z,x}^s$ is precisely $\Gamma$-invariant in the sense that $g^*m_{g(z),x}^s=m_{z,x}^s$ for every $g \in \Gamma$.
Indeed,
$$
g^*m_{g(z),x}^s=\frac{1}{P_\Gamma^s(o,x)}\sum_{\gamma \in \Gamma} a^{-sd(g(z),\gamma(x))} D_{g^{-1}\gamma(x)}
=\frac{1}{P_\Gamma^s(o,x)}\sum_{\gamma \in \Gamma} a^{-sd(g(z),g\gamma(x))} D_{\gamma(x)}=m_{z,x}^s.
$$

For any decreasing sequence of $s$ to $e_a(\Gamma)$, there is a subsequence $\{s_i\}_{i \in \N}$ such that
$m_{z,x}^{s_i}$ converge to some measure on $\overline X$ in the weak-$\ast$ sense.
We denote this limit measure by $m_{z,x}$, even though it also depends on the choice of the sequence $\{s_i\}$.
However, $m_{z,x}^s$ is invariant when $x$ is replaced in the orbit $\Gamma(x)$; furthermore, it is
$\Gamma$-invariant, as shown above. Therefore,
we can take the same sequence $\{s_i\}$ for all $\gamma(x)$ and for all $\gamma(z)$ $(\gamma \in \Gamma)$; we assume this choice hereafter.
The total mass $\Vert m_{z,x} \Vert$ satisfies the same inequalities for $\Vert m_{z,x}^s \Vert$ after replacing $s$ with $e_a(\Gamma)$
and, in particular, $\Vert m_{o,x} \Vert=1$.
Owing to the condition that $\Gamma$ is of divergence type, it can be proved that the support of $m_{z,x}$ 
is in the limit set $\Lambda(\Gamma)$. 

By fixing any orbit point $x$, we have $\{m_{z,x}\}_{z \in X}$,
which we call the {\it canonical measure family}. 
When the orbit point $x$ is not in question or is assumed to be the base point $o$, we denote the canonical measure family
by $\{m_z\}_{z \in X}$ for brevity.
As $\{m_{z,x}^s\}_{z \in X}$ is 
$\Gamma$-invariant, so is the canonical measure family $\{m_{z,x}\}_{z \in X}$.

We show that this is a Patterson measure family; we can call it the {\it canonical Patterson measure family} hereafter.
The proof is a modification of that in Coornaert \cite[Th\'eor\`eme 5.4]{Co}.

\begin{lemma}\label{canonical}
Let $\Gamma \subset \Isom(X,d)$ is a discrete group of divergence type.
For any $x \in X$, the canonical measure family $\{m_{z,x}\}_{z \in X}$ is a quasiconformal measure family of dimension $e_a(\Gamma)$ with 
quasiconformal constant $a^{c(\delta)} \geq 1$.  
Hence, this is a
Patterson measure family for $\Gamma$ with quasi-invariance constant $1$.
\end{lemma}

\begin{proof}
For every $\xi \in \partial X$ and for every $z \in X$,
we choose a neighborhood $U_\xi \subset \overline X$ of $\xi$
as in Proposition \ref{approx}. Let $f$ be any continuous function on $U_\xi$ with compact support. Then
\begin{align*}
m_{z,x}^s(f)&=\int_{U_\xi}f(\zeta)dm_{z,x}^s(\zeta)=\frac{1}{P_\Gamma^s(o,x)}\sum_{\gamma(x) \in U_\xi} a^{-sd(z,\gamma(x))}f(\gamma(x));\\
m_{o,x}^s(f)&=\int_{U_\xi}f(\zeta)dm_{o,x}^s(\zeta)=\frac{1}{P_\Gamma^s(o,x)}\sum_{\gamma(x) \in U_\xi} a^{-sd(o,\gamma(x))}f(\gamma(x)).
\end{align*}
It follows from Proposition \ref{approx} with the constant $c(\delta)$ that  
$$
a^{-sc(\delta)}k(z,\xi)^s \leq \frac{m_{z,x}^s(f)}{m_{o,x}^s(f)} \leq
a^{sc(\delta)}k(z,\xi)^s.
$$
Taking the limit of some subsequence $\{s_i\}$ as $s \to e_a(\Gamma)=e$, 
which may be different for $m_{z,x}^s$ and for $m_{o,x}^s$, we have
$$
a^{-ec(\delta)}k(z,\xi)^{e} \leq \frac{m_{z,x}(f)}{m_{o,x}(f)} \leq a^{ec(\delta)}k(z,\xi)^{e}.
$$
As $f$ is arbitrary, this implies that
$(dm_{z,x}/dm_{o,x})(\xi) \asymp_{a^{ec(\delta)}} k(z,\xi)^e$. Thus, the quasiconformality with constant $a^{c(\delta)}$ is proved.
\end{proof}

\begin{remark}
Regarding the canonical Patterson measure family $\{m_{z,x}\}_{z \in X}$,
if we consider $\mu=m_{o,x}$, then 
by Proposition \ref{singlemeasure},
$\mu$ is a $(\Gamma,a^{e_a(\Gamma)c(\delta)})$-quasi-invariant quasiconformal measure with total mass $\Vert \mu \Vert=1$. 
We also call this the {\it canonical Patterson measure}. It is that given in \cite[Th\'eor\`eme 5.4]{Co},
where the convergence-type group case was also treated.
\end{remark}

The canonical Patterson measure family is quasi-unique in the sense that it is independent of the choice of 
the orbit point $x \in X$ and the weak-$\ast$ limit.

\begin{lemma}\label{uniquecanonical}
The canonical Patterson measure families $\{m_{z,x}\}_{z \in X}$ and $\{m_{z,x'}\}_{z \in X}$ for $x,x' \in X$ satisfy
$(dm_{z,x'}/dm_{z,x})(\xi) \asymp_{K} 1$
for $K=a^{4e_a(\Gamma)c(\delta)}$.
This includes the case where the weak-$\ast$ limits $m_{z,x}$ and $m_{z,x'}$ are different even if $x=x'$.
\end{lemma}

\begin{proof}
We consider the canonical Patterson measures $\mu=m_{o,x}$ and $\mu'=m_{o,x'}$,
which are $(\Gamma,a^{e_a(\Gamma)c(\delta)})$-quasi-invariant as in the remark above.
As $\Vert \mu \Vert=\Vert \mu' \Vert=1$, Theorem \ref{uniqueness} implies that $(d\mu'/d\mu)(\xi) \asymp_{a^{2e_a(\Gamma)c(\delta)}}1$.
Moreover, the quasiconformality by Lemma \ref{canonical} yields that
$$
\frac{dm_{z,x'}}{dm_{z,x}}(\xi) \asymp_{a^{2e_a(\Gamma)c(\delta)}} \frac{dm_{o,x'}}{dm_{o,x}}(\xi).
$$
This proves the statement.
\end{proof}

Now we consider the quasi-invariance of the Patterson measure for $\Gamma$ under its normalizer. For a subgroup $\Gamma \subset \Isom(X,d)$,
this is denoted by
$$
N(\Gamma)=\{g \in \Isom(X,d) \mid g \Gamma g^{-1}=\Gamma\}.
$$
First, we consider the pull-back of a $\Gamma$-quasi-invariant quasiconformal measure family by each element $g \in N(\Gamma)$.

\begin{proposition}\label{pullback}
For a discrete group $\Gamma \subset \Isom(X,d)$, let $\{\mu_z\}_{z \in X}$ be an $s$-dimensional $(\Gamma,D)$-quasi-invariant
quasiconformal measure family with quasiconformal constant $C \geq 1$. Then, $\{g^*\mu_{g(z)}\}_{z \in X}$
is also an $s$-dimensional $(\Gamma,D)$-quasi-invariant
quasiconformal measure family with quasiconformal constant $C' \geq 1$
for every $g \in N(\Gamma)$. Here $C'=a^{4\kappa(\delta)}C^2$; in particular, it is independent of $g \in N(\Gamma)$.
\end{proposition}

\begin{proof}
Let $\nu_z=g^*\mu_{g(z)}$ for brevity. We first prove that $\{\nu_z\}_{z \in X}$ is a quasiconformal measure family.
This is done by
$$
\frac{d\nu_z}{d\nu_o}(\xi)=\frac{d\mu_{g(z)}}{d\mu_{g(o)}}(g(\xi))
\asymp_{C^{2s}} \frac{k(g(z),g(\xi))^s}{k(g(o),g(\xi))^s}
\asymp_{a^{4s\kappa(\delta)}} k(z,\xi)^s \quad (\almostall \xi \in \partial X),
$$
where the last estimate follows from Proposition \ref{poisson}. The quasiconformal constant $C'$ can be chosen as
$C'=a^{4\kappa(\delta)}C^2$.
To see that $\{\nu_z\}_{z \in X}$ is $(\Gamma,D)$-quasi-invariant, we take an arbitrary $\gamma \in \Gamma$ and
its conjugate $\tilde \gamma \in \Gamma$ satisfying $g\gamma=\tilde \gamma g$. Then,
$$
\gamma^*\nu_{\gamma(z)}=\gamma^*g^*\mu_{g\gamma(z)}=g^* \tilde \gamma^*\mu_{\tilde\gamma g(z)} \asymp_D g^*\mu_{g(z)}=\nu_z
$$
yields the desired condition.
\end{proof}

Lemma \ref{canonical} and Proposition \ref{pullback} show the following consequence.
Hereafter, if a Patterson measure family is $(\Gamma, D)$-quasi-invariant, then
we call it a $(\Gamma, D)$-Patterson measure family for brevity.

\begin{corollary}\label{cor}
Let $\{m_{z}\}_{z \in X}$ be the canonical Patterson measure family for a non-ele\-men\-tary discrete group $\Gamma$ of divergence type.
Then, for every $g \in N(\Gamma)$, $\{g^*m_{g(z)}\}_{z \in X}$ is a $(\Gamma,1)$-Patterson measure family 
with quasiconformal constant $a^{4\kappa(\delta)+2c(\delta)}$.
\end{corollary}

We are now ready to explain our main result in this section. For a discrete group $\Gamma \subset \Isom(X,d)$ of divergence type,
we take a Patterson measure family $\{\mu_z\}_{z \in X}$ for $\Gamma$.  
Then, by Proposition \ref{pullback}, 
every $g \in N(\Gamma)$ yields a Patterson measure family $\{g^*\mu_{g(z)}\}_{z \in X}$ for $\Gamma$. 
Owing to the quasi-uniqueness by Theorem \ref{uniqueness}, this is comparable with the original $\{\mu_z\}_{z \in X}$.
If this comparison is uniform independently of $g \in N(\Gamma)$, then we can conclude that $\{\mu_z\}_{z \in X}$
is quasi-invariant under $N(\Gamma)$. The problem is to show this uniformity; more precisely,
to estimate the total mass of $\{g^*\mu_{g(z)}\}_{z \in X}$. To this end, we have to utilize the canonical Patterson measure family
$\{m_z\}_{z \in X}$ rather than $\{\mu_z\}_{z \in X}$.

\begin{lemma}\label{totalmass}
Let $\{m_{z,x}\}_{z \in X}$ be any canonical Patterson measure family for a non-elemen\-tary
discrete group $\Gamma$ of divergence type.
Then, for every $g \in N(\Gamma)$ and for every $x \in X$,
the total mass of $m_{g(o),x}$ satisfies $\Vert m_{g(o),x} \Vert \asymp_{K^{3/2}} 1$,
where $K=a^{4e_a(\Gamma)c(\delta)}$ is the constant given in Lemma \ref{uniquecanonical}.
\end{lemma}

\begin{proof}
Suppose that $m_{g(o),o}$ is the weak-$\ast$ limit of $m_{g(o),o}^{s_i}$ with $s_i \searrow e=e_a(\Gamma)$
as $i \to \infty$. Then,
$$
\Vert m_{g(o),o} \Vert=\lim_{i \to \infty} \Vert m_{g(o),o}^{s_i} \Vert
=\lim_{i \to \infty} \frac{P_\Gamma^{s_i}(g(o),o)}{P_\Gamma^{s_i}(o,o)}.
$$
By Proposition \ref{property}, this ratio of the Poincar\'e series can be represented as
$$
\frac{P_\Gamma^{s_i}(g(o),o)}{P_\Gamma^{s_i}(o,o)}
=\left(\frac{P_\Gamma^{s_i}(g(o),g(o))}{P_\Gamma^{s_i}(o,g(o))}\right)^{-1}.
$$

We choose a subsequence of $\{s_i\}$ (denoted by the same $s_i$) so that 
$m_{g(o),g(o)}^{s_i}$ converge to some
$m_{g(o),g(o)}$ in the weak-$\ast$ sense.
Then, the above ratio of the Poincar\'e series converges to  
$\Vert m_{g(o),g(o)} \Vert^{-1}$ as $i \to \infty$. This shows that $\Vert m_{g(o),o} \Vert=\Vert m_{g(o),g(o)} \Vert^{-1}$.
On the contrary, Lemma \ref{uniquecanonical} implies that
$\Vert m_{g(o),o} \Vert \asymp_{K}\Vert m_{g(o),g(o)}\Vert$. Hence, we have $\Vert m_{g(o),o} \Vert \asymp_{K^{1/2}} 1$.
By Lemma \ref{uniquecanonical} again, $\Vert m_{g(o),x} \Vert \asymp_{K^{3/2}} 1$.
\end{proof}

Our main result, the quasi-invariance of the Patterson measure under the normalizer, is formulated as follows.

\begin{theorem}\label{normalizer}
Let $\Gamma \subset \Isom(X,d)$ be a non-elementary discrete group of divergence type and
let $\{\mu_z\}_{z \in X}$ be a $(\Gamma,D_0)$-Patterson measure family
with quasiconformal constant $C_0$. Then, there exists a
constant $D \geq 1$ depending only on $C_0$, $D_0$, $\delta$, $a$, and $e_a(\Gamma)$ such that
$$
D^{-1} \leq \frac{dg^*\mu_{g(z)}}{d\mu_z}(\xi) \leq D \quad (\almostall \xi \in \partial X)
$$
for every $g \in N(\Gamma)$.
\end{theorem}

\begin{proof}
We first prove the result for the canonical Patterson measure family $\{m_z\}_{z \in X}$.
By Corollary \ref{cor}, $\{g^*m_{g(z)}\}_{z \in X}$ is a $(\Gamma,1)$-Patterson measure family 
with the quasiconformal constant $C=a^{4\kappa(\delta)+2c(\delta)}$.
Let $\nu_o=g^*m_{g(o)}$, which is a $(\Gamma,C^{e_a(\Gamma)})$-Patterson measure by 
the remark after Lemma \ref{canonical}.
Its total mass is $\Vert \nu_o \Vert=\Vert m_{g(o)} \Vert \asymp_{K^{1/2}} 1$ for $K=a^{4e_a(\Gamma)c(\delta)}$
by the proof of Lemma \ref{totalmass}. We already know that $m_o$ is also a $(\Gamma,C'^{e_a(\Gamma)})$-Patterson measure
with $\Vert m_o \Vert=1$
by the same remark mentioned above, where we set $C'=a^{c(\delta)}$.
Then, Theorem \ref{uniqueness} asserts that
$$
(CC')^{-e_a(\Gamma)}K^{-1/2} \leq \frac{d\nu_o}{dm_o}(\xi)=\frac{dg^*m_{g(o)}}{dm_o}(\xi)\leq (CC')^{e_a(\Gamma)}K^{1/2}
\quad (\almostall \xi \in \partial X).
$$
Finally, by using the quasiconformality of $\{m_z\}_{z \in X}$ with the constant $C'=a^{c(\delta)}$
and $\{g^*m_{g(z)}\}_{z \in X}$ with the constant $C=a^{4\kappa(\delta)+2c(\delta)}$, we have
$$
(CC')^{-2e_a(\Gamma)}K^{-1/2}
\leq \frac{dg^*m_{g(z)}}{dm_z}(\xi) \leq 
(CC')^{2e_a(\Gamma)}K^{1/2}.
$$

We now consider a $(\Gamma,D_0)$-Patterson measure family $\{\mu_z\}_{z \in X}$ in general
with quasiconformal constant $C_0$. For the sake of simplicity,
we may assume that $\Vert \mu_o \Vert=1$. By Proposition \ref{singlemeasure}, $\mu_o$ is
$(\Gamma,C_0^{e_a(\Gamma)}D_0)$-quasi-invariant, and then Theorem \ref{uniqueness} gives
$$
\frac{d\mu_o}{dm_o}(\xi) \asymp_{(C_0 C')^{e_a(\Gamma)}D_0} 1.
$$
Similar to the process above, the quasiconformality then yields
$$
(C_0 C')^{-2e_a(\Gamma)}D_0^{-1}
\leq \frac{d\mu_z}{dm_z}(\xi) \leq 
(C_0 C')^{2e_a(\Gamma)}D_0.
$$
By replacing $z$ with $g(z)$ and $\xi$ with $g(\xi)$ here, we also see that
$(dg^*\mu_{g(z)}/dg^*m_{g(z)})(\xi)$ is bounded from above and below by the same constants.
Hence, the above three inequalities conclude that
$$
D^{-1} \leq \frac{dg^*\mu_{g(z)}}{d\mu_z}(\xi) \leq D \quad (\almostall \xi \in \partial X)
$$
for $D= (CC')^{2e_a(\Gamma)}K^{1/2} (C_0 C')^{4e_a(\Gamma)}D_0^2$.
\end{proof}

\medskip
\section{No proper conjugation for divergence-type groups}\label{7}
In this section, we consider the proper conjugation problem for discrete 
isometry groups of the Gromov hyperbolic space $(X,d)$.
This is a continuation of our previous work \cite{MY1, MY2}, where we proved the corresponding results for
Kleinian groups of divergence type and convex cocompact subgroups of $\Isom(X,d)$.
A history of this problem and preceding results can be found in \cite{MY1, MY2} and the references therein.

First, we mention an assumption in our new theorem, 
which was not necessary in the previous theorems.
For Kleinian groups, the J{\o}rgensen theorem ensures
that the geometric limit of a sequence of discrete groups
is also discrete. To avoid these problems in the present arguments for discrete subgroups of $\Isom(X,d)$ that are not necessarily convex cocompact,
we introduce the following additional assumption. This was already mentioned in \cite{MY2}.

\begin{definition}
We say that a discrete group $\Gamma \subset \Isom(X,d)$ is {\it uniformly properly discontinuous} if there are 
a constant $r >0$ and a positive integer $N \in \N$ such that
the number of elements $\gamma \in \Gamma$ satisfying 
$\gamma(B(x,r)) \cap B(x,r)\neq \emptyset$ is bounded by $N$ for every $x \in X$. 
\end{definition}

We prepare some claims that are used in the arguments below.
For a sequence of discrete subgroups $\{\Gamma_n\}$ of $\Isom(X,d)$, we define the {\it envelope} denoted by
${\rm Env}\{\Gamma_n\}$ to be the subgroup of $\Isom(X,d)$ consisting of all elements 
$\gamma=\lim_{n \to \infty} \gamma_n$
given for some sequence $\gamma_n \in \Gamma_n$.
We recall the following fact as in \cite[Proposition 2.4]{MY2}. 

\begin{proposition}\label{uniformdisconti}
Let $\{\Gamma_n\}_{n=1}^\infty$ be a sequence of subgroups of $\Isom(X,d)$ that act
uniformly properly discontinuously on $X$ where the uniformity is also independent
of $n$. Then, ${\rm Env}\{\Gamma_n\}$ also acts uniformly properly discontinuously on $X$.
\end{proposition}

In addition, lower semi-continuity of the critical exponents, which is known to be true for
geometric convergence of Kleinian groups, is valid in the following form.

\begin{proposition}\label{semiconti}
Let $\{\Gamma_n\}_{n=1}^\infty \subset \Isom(X,d)$ be a sequence of discrete groups of divergence type and let
$\Gamma_\infty$ be a discrete subgroup of ${\rm Env}\{\Gamma_n\}$. Then
$$
\liminf_{n \to \infty} e_a(\Gamma_n) \geq e_a(\Gamma_\infty).
$$
\end{proposition}

\begin{proof}
Let $e=\liminf_{n \to \infty} e_a(\Gamma_n)$. For each $\Gamma_n$, we take the canonical Patterson measure 
$\mu_n=(m_o)_n$.
Passing to a subsequence, we may assume that both $e_a(\Gamma_n)$ converge to $e<\infty$ and 
$\mu_n$ converge to some Borel measure $\mu$ on $\partial X$ with $\Vert \mu \Vert=1$ in the weak-$\ast$ sense as $n \to \infty$.
Here, we see that $\mu$ is a $(\Gamma_\infty,a^{ec(\delta)})$-quasi-invariant quasiconformal measure of dimension $e$. Indeed,
each canonical Patterson measure $\mu_n$ is $(\Gamma_n,a^{e_a(\Gamma_n)c(\delta)})$-quasi-invariant 
by the remark after Lemma \ref{canonical}, and the weak-$\ast$ limit $\mu$
preserves this quasi-invariance for the group ${\rm Env}\{\Gamma_n\}$ and for the dimension $e$. By Theorem \ref{existence},
the existence of such a measure $\mu$ for $\Gamma_\infty$ yields $e \geq e_a(\Gamma_\infty)$.
\end{proof}

The quasi-invariance of the Patterson measure under the normalizer (Theorem \ref{normalizer})
will be used in the following situation.
Although there is no essential difference, this slightly generalized formulation is more convenient.

\begin{proposition}\label{morenormal}
Let $\Gamma$ and $\widetilde \Gamma$ be non-elementary discrete groups of divergence type in $\Isom(X,d)$
such that $\Gamma \subset \widetilde \Gamma$ and $e_a(\Gamma)=e_a(\widetilde \Gamma)$. Then, 
a Patterson measure (family) for $\Gamma$ is quasi-invariant under the normalizer
$N(\widetilde \Gamma)$ of $\widetilde \Gamma$. More precisely, if $\{\mu_z\}_{z \in X}$ is a $(\Gamma,D_0)$-Patterson measure family
with quasiconformal constant $C_0$, then there exists a
constant $D \geq 1$ depending only on $C_0$, $D_0$, $\delta$, $a$, and $e_a(\Gamma)$ such that
$$
D^{-1} \leq \frac{dg^*\mu_{g(z)}}{d\mu_z}(\xi) \leq D \quad (\almostall \xi \in \partial X)
$$
for every $g \in N(\widetilde \Gamma)$.
\end{proposition}

\begin{proof}
We take the canonical Patterson measure family $\{\widetilde \mu_z\}_{z \in X}$ for $\widetilde \Gamma$,
which is $(\widetilde \Gamma,1)$-quasi-invariant with quasiconformal constant $a^{c(\delta)}$ by Lemma \ref{canonical}.
Then, this is quasi-invariant under $N(\widetilde \Gamma)$ as in Theorem \ref{normalizer}.
Furthermore, because $\Gamma \subset \widetilde \Gamma$, $\{\widetilde \mu_z\}_{z \in X}$ is a
$(\Gamma,1)$-Patterson measure family with the quasiconformal constant $a^{c(\delta)}$.

Let $\{\mu_z\}_{z \in X}$ be a $(\Gamma,D_0)$-Patterson measure family
with quasiconformal constant $C_0$. We may assume that $\Vert \mu_o \Vert=1$.
By Theorem \ref{uniqueness} with the remark in the previous section, we have
$(d\mu_o/d\widetilde \mu_o)(\xi) \asymp_{D_0(a^{c(\delta)} C_0)^{e_a(\Gamma)}} 1$.
Then, 
$$
\frac{d\mu_z}{d\widetilde \mu_z}(\xi) \asymp_{D_0(a^{c(\delta)} C_0)^{2e_a(\Gamma)}} 1 \quad (\almostall \xi \in \partial X)
$$ 
for every $z \in X$. By the quasi-invariance of $\{\widetilde \mu_z\}_{z \in X}$ under $N(\widetilde \Gamma)$, 
$\{\mu_z\}_{z \in X}$ is also $N(\widetilde \Gamma)$-quasi-invariant. Moreover,
the dependence of the constant $D$ is as stated.
\end{proof}

We state and prove the main theorem in this section. We say that $G \subset \Isom(X,d)$ admits {\it proper conjugation}
if there is some element $\alpha \in \Isom(X,d)$ such that the conjugate $\alpha G \alpha^{-1}$ is a proper subgroup of $G$.
Our result says that divergence-type groups do not permit such an unusual conjugation.

\begin{theorem}\label{noproper}
Let $G \subset \Isom(X,d)$ be a non-elementary discrete group of divergence type
that is uniformly properly discontinuous. If 
$\alpha G \alpha^{-1} \subset G$ for $\alpha \in \Isom(X,d)$, then $\alpha G \alpha^{-1} =G$.
\end{theorem}

\begin{proof}
Let $\Gamma=\alpha G \alpha^{-1}$ and $\Gamma_n=\alpha^{-n}\Gamma \alpha^n$ for each integer $n \geq 0$. Then
$\Gamma_0=\Gamma$, $\Gamma_1=G$, and $\{\Gamma_n\}_{n \geq 0}$ is an increasing sequence of discrete subgroups of
$\Isom(X,d)$ that are
conjugate to $G$. In particular, they are all uniformly properly discontinuous and they are of divergence type
with the same critical exponent
$e=e_a(G)$. We define $\Gamma_\infty=\bigcup_{n \geq 0} \Gamma_n$, which coincides with the envelope ${\rm Env}\{\Gamma_n\}$
in this case. By Proposition \ref{uniformdisconti}, $\Gamma_\infty$ is a discrete subgroup.
As $e_a(\Gamma_n)=e$, Proposition \ref{semiconti} implies that $e_a(\Gamma_\infty) \leq e$. However, as
the converse inequality is trivial by the inclusion relation of groups, we have $e_a(\Gamma_\infty)=e$.
Moreover, $\Gamma_\infty$ is clearly of divergence type because it includes $\Gamma_n$. Furthermore, as the limit of $\Gamma_{n-1}=\alpha \Gamma_n \alpha^{-1} \subset \Gamma_n$, we have
$\alpha \Gamma_\infty \alpha^{-1}=\Gamma_\infty$; 
thus, $\alpha \in N(\Gamma_\infty)$.

To prove the statement, we suppose to the contrary that $\Gamma \subsetneqq G$ and set $\ell=[G:\Gamma] \in [2,\infty]$. Let
$$
G=g_1 \Gamma \sqcup g_2 \Gamma \sqcup \cdots \sqcup g_k \Gamma \sqcup \cdots
$$
be a coset decomposition of $G$ by $\Gamma$. Accordingly, we decompose the weighted Dirac measures 
$(m_G)^s_{o,o}$ given by the Poincar\'e series $P_G^s(o,o)$ for $s>e$ to be
$(m_G)^s_{o,o}=\sum_{k=1}^\ell \nu_k$, where
$$
\nu_k^s=\frac{1}{P_G^s(o,o)}\sum_{\gamma \in \Gamma} a^{-s d(o,g_k\gamma(o))}D_{g_k \gamma(o)}.
$$
Using the weighted Dirac measures 
$(m_\Gamma)^s_{g_k^{-1}(o),o}$ given by the Poincar\'e series $P_\Gamma^s(g_k^{-1}(o),o)$, we represent $\nu_k^s$ by
$$
\nu_k^s=\frac{P_\Gamma^s(o,o)}{P_G^s(o,o)} \frac{1}{P_\Gamma^s(o,o)}
\sum_{\gamma \in \Gamma} a^{-s d(g_k^{-1}(o),\gamma(o))}(g_k^{-1})^*D_{\gamma(o)}
=\frac{P_\Gamma^s(o,o)}{P_G^s(o,o)}\, (g_k^{-1})^*(m_\Gamma)^s_{g_k^{-1}(o),o}.
$$
Moreover, by $\Gamma=\alpha G \alpha^{-1}$ and Proposition \ref{property}, we have
$$
\frac{P_\Gamma^s(o,o)}{P_G^s(o,o)}=\frac{P_G^s(\alpha^{-1}(o),\alpha^{-1}(o))}{P_G^s(o,o)}
=\frac{P_G^s(\alpha^{-1}(o),o)}{P_G^s(o,o)}\cdot \frac{P_G^s(\alpha^{-1}(o),\alpha^{-1}(o))}{P_G^s(o,\alpha^{-1}(o))}.
$$
The corresponding substitutions yield 
$$
(m_G)^s_{o,o}=\sum_{k=1}^\ell \frac{P_G^s(\alpha^{-1}(o),o)}{P_G^s(o,o)}\cdot 
\frac{P_G^s(\alpha^{-1}(o),\alpha^{-1}(o))}{P_G^s(o,\alpha^{-1}(o))} (g_k^{-1})^*(m_\Gamma)^s_{g_k^{-1}(o),o}.
$$

We take the limit of the above equality. We can choose a sequence $s_i \searrow e$ such that all the involved
terms are convergent because they are at most countably many. As a result, we obtain
$$
(m_G)_{o,o} \geq \sum_{k=1}^\ell \Vert (m_G)_{\alpha^{-1}(o),o} \Vert \cdot \Vert (m_G)_{\alpha^{-1}(o),\alpha^{-1}(o)} \Vert\,
(g_k^{-1})^*(m_\Gamma)_{g_k^{-1}(o),o},
$$
where $\{(m_G)_{z,x}\}$ and $\{(m_\Gamma)_{z,x}\}$ stand for the canonical Patterson measure families for $G$ and $\Gamma$, respectively.
Here, we use Proposition \ref{morenormal} for $\widetilde \Gamma=\Gamma_\infty$. Then, there is a constant $D \geq 1$
independent of the elements of $N(\Gamma_\infty)$ such that
$$
D^{-1} \leq \frac{d(g_k^{-1})^*(m_\Gamma)_{g_k^{-1}(o),o}}{d(m_\Gamma)_{o,o}}(\xi) 
\leq D \quad (\almostall \xi \in \partial X).
$$
In particular, the total mass satisfies $\Vert (g_k^{-1})^*(m_\Gamma)_{g_k^{-1}(o),o} \Vert \asymp_D 1$.
Similarly, we have
\begin{align*}
&\quad \Vert (m_G)_{\alpha^{-1}(o),o} \Vert=\Vert (\alpha^{-1})^*(m_G)_{\alpha^{-1}(o),o} \Vert \asymp_D 
\Vert (m_G)_{o,o} \Vert=1;\\
&\quad \Vert (m_G)_{\alpha^{-1}(o),\alpha^{-1}(o)} \Vert=\Vert (\alpha^{-1})^*(m_G)_{\alpha^{-1}(o),\alpha^{-1}(o)} \Vert \asymp_D 
\Vert (m_G)_{o,\alpha^{-1}(o)} \Vert=1.
\end{align*}
Then, taking the total mass in the above inequality, we can make the assertion 
$\ell=[G:\Gamma] \leq D^3$. If $\ell=\infty$, this is a contradiction; we may assume that $\ell<\infty$.

Finally, we choose $j \in \N$ such that $\ell^j>D^3$.
We consider $\alpha^j$ instead of $\alpha$ and set $\Gamma'=\alpha^j G \alpha^{-j}$, which is a proper subgroup of $G$
with index $[G:\Gamma']=\ell^j$. Then, we repeat the same arguments as above for $G$ and $\Gamma'$.
The conclusion is that $[G:\Gamma'] \leq D^3$. We note that the constant $D$ is unaffected by this replacement
because the dependence of $D$ as in Proposition \ref{morenormal} is irrelevant to the canonical Patterson measures.
In this way, we derive the contradiction, and thus prove the result.
\end{proof}
\medskip

\section{The lower bound of the critical exponents of normal subgroups}\label{8}

For Kleinian groups, there are numerous important studies on the critical exponents of non-elementary normal subgroups $\Gamma$.
Among them, concerning the lower bound of such exponents, Falk and Stratmann \cite{FS} proved that they are bounded from below
by half the exponent of the original group $G$. Later, Roblin \cite{R2} extended this result in a different manner and 
it was proved, in particular, that if $G$ is of divergence type, then the strict inequality holds.
More recently, a simple proof for these results appeared in \cite{J}.
We generalize this argument to discrete isometry groups of the Gromov hyperbolic space $\Isom(X,d)$ 
and prove the following theorem.

\begin{theorem}\label{main3}
Let $G \subset \Isom(X,d)$ be a discrete group and let $\Gamma$ be a non-elementary normal subgroup of $G$.
Then, $e_a(\Gamma) \geq e_a(G)/2$. Moreover, if $G$ is of divergence type, then the strict inequality
$e_a(\Gamma) > e_a(G)/2$ holds.
\end{theorem}

This theorem was already expected in \cite{J} when
\cite[Theorem 4.3]{MY1}, which was used for the proof of the strict inequality,
was going to be generalized to the case of the Gromov hyperbolic space.
This generalization is here carried out as a consequence of Theorem \ref{normalizer} in the following form.

\begin{theorem}\label{normal}
Let $G \subset \Isom(X,d)$ be a discrete group and let $\Gamma$ be a non-elementary normal subgroup of $G$.
If $\Gamma$ is of divergence type, then $e_a(G)=e_a(\Gamma)$; moreover, $G$ is also of divergence type.
\end{theorem}

\begin{proof}
Let $\mu$ be a Patterson measure for $\Gamma$. By Theorem \ref{normalizer}, 
$\mu$ is quasi-invariant under $N(\Gamma)$.
In particular, $\mu$ is a $G$-quasi-invariant quasiconformal measure of dimension $e_a(\Gamma)$.
On the contrary, by Theorem \ref{existence}, the lower bound of the dimensions of $G$-quasi-invariant quasiconformal measures
is $e_a(G)$. Hence, we have $e_a(\Gamma) \geq e_a(G)$. As the converse inequality is trivial by $\Gamma \subset G$, we see that
$e_a(G)=e_a(\Gamma)$. Moreover, the divergence of $\Gamma$ at $e_a(\Gamma)$ implies that of $G$ at the same dimension.
\end{proof}

\begin{remark}
Corollary \ref{inclusion} asserts that when $\Gamma$ is a subgroup of $G$ not necessarily normal and $\Gamma$ is of divergence type,
$e_a(G)=e_a(\Gamma)$ implies $\Lambda(G)=\Lambda(\Gamma)$. The converse is not true in general, but Theorem \ref{normal}
states that if $\Gamma$ is non-elementary and normal in $G$, which implies $\Lambda(G)=\Lambda(\Gamma)$, then $e_a(G)=e_a(\Gamma)$.
\end{remark}

We add necessary modification to the claims in \cite{J} to apply them to discrete isometry groups of
the Gromov hyperbolic space $(X,d)$.

All non-trivial elements of $\Isom(X,d)$ are classified into three types: hyperbolic, parabolic, and elliptic.
We say that $\gamma \in \Isom(X,d)$ is {\it hyperbolic} if it has exactly two fixed points on $\partial X$.
The following are well-known properties of hyperbolic elements of discrete groups: see, for example, Tukia \cite[Section 2]{T2}.
We note that $\Isom(X,d)$ acts on the boundary $\partial X$ as a convergence group.

\begin{proposition}\label{facts}
Let $\Gamma \subset \Isom(X,d)$ be a non-elementary discrete group. Then, $\Gamma$ contains a hyperbolic element $h$.
Moreover, the stabilizer $\Stab_\Gamma(\Fix(h))$ of the fixed point set $\Fix(h) \subset \partial X$ 
of $h$ is a finite index extension of the cyclic group $\langle h \rangle$. If $\gamma \in \Gamma$ commutes with $h$, then
$\gamma$ belongs to $\Stab_\Gamma(\Fix(h))$.
\end{proposition}

The novelty of the proof in \cite{J} is the use of the following fact. Once the above properties are verified,
the proof of the lemma can be carried out without any change even in the case of the Gromov hyperbolic space.

\begin{lemma}\label{finite-to-one}
Let $G \subset \Isom(X,d)$ be a discrete group and let $\Gamma$ be a non-elementary normal subgroup of $G$.
For any hyperbolic element $h \in \Gamma$, the map
$$
\iota_h:\langle h \rangle \backslash G \to \Gamma
$$
defined by $[g] \mapsto g^{-1}hg$ is well-defined and at most $k$ to $1$, that is,
there is $k=k_h \in \N$ such that 
$\# \iota_h^{-1}(\gamma) \leq k_h$ for every $\gamma \in \Gamma$. 
\end{lemma}

An essential step in the adaption of the arguments for Kleinian groups 
to discrete isometry groups on the Gromov hyperbolic space lies in the following claim.

\begin{lemma}\label{essential}
Let $G \subset \Isom(X,d)$ be a discrete group and let $h \in G$ be a hyperbolic element.
Then, for every $s>0$, there is a constant $A_h(s)>0$ depending on $s$ and $h$ such that
$$
\sum_{g \in G} a^{-sd(o,g(o))} \leq A_h(s) \sum_{[g] \in \langle h \rangle \backslash G} a^{-sd(o,[g](o))},
$$
where $d(o,[g](o))$ is the distance from $o$ to the set $[g](o)=\{h^ng(o)\mid n \in \Z\}$.
\end{lemma}

\begin{proof}
We take a geodesic segment $[o,h(o)]$ connecting $o$ and $h(o)$ 
and make a piecewise geodesic curve $\beta=\bigcup_{n \in \Z} h^n([o,h(o)])$ with arc length parameter.
In fact, $\beta:(-\infty,\infty) \to X$ is a quasi-geodesic line, that is, there are constants
$\lambda \geq 1$ and $c \geq 0$ such that 
$$
|s-t|\leq \lambda d(\beta(s),\beta(t))+c
$$ 
for all $s,t \in \R$
(see \cite[Lemme 6.5]{CDP}).
Let $\ell_h=d(o,h(o))$.

For any coset $[g] \in \langle h \rangle \backslash G$, we consider the set $[g](o)$ and
choose a point in it, which we may assume to be $g(o)$ without loss of generality, so that
$[o,h(o)]$ contains the nearest point $x$ (not necessarily unique) from $g(o)$ to $\beta$.
Let $L_{[g]}=d(x,g(o))=d(\beta,[g](o))$. Then,
we have
$$
d(o,[g](o)) \leq d(o,g(o)) \leq d(o,x)+d(x,g(o)) \leq \ell_h+L_{[g]}.
$$

We consider $h^ng(o)$ for every $n \in \Z$. By the invariance under $\langle h \rangle$, 
$h^n(x)$ is the nearest point from $h^ng(o)$ to $\beta$. The above inequality implies that
$$
d(h^n(x),h^ng(o))=L_{[g]} \geq d(o,[g](o))-\ell_h.
$$
We choose a geodesic segment $\widetilde \beta_n=[o,h^n(x)]$. As $\widetilde \beta_n$ is within distance 
$r=r(\delta,\lambda,c) \geq 0$
from the quasi-geodesic segment in $\beta$ between $o$ and $h^n(x)$  
(see \cite[Th\'eor\`eme 1.3]{CDP}), 
we have
$$
d(\widetilde \beta_n,h^ng(o)) \geq d(h^n(x),h^ng(o))-r=L_{[g]}-r.
$$ 
On the contrary, the Gromov product satisfies
$$
(o \mid h^n(x))_{h^ng(o)}=\frac{1}{2}\{d(o,h^ng(o))+d(h^n(x),h^ng(o))-d(o,h^n(x))\} \geq d(\widetilde \beta_n,h^ng(o))-4\delta.
$$
These inequalities imply that
\begin{align*}
d(o,h^ng(o)) &\geq d(o,h^n(x))+d(h^n(x),h^ng(o))-2r-8\delta\\
&\geq d(o,h^n(o))-\ell_h+L_{[g]}-2r-8\delta\\
&\geq d(o,h^n(o))+d(o,[g](o))-2(\ell_h+r+4\delta).
\end{align*}

Using this estimate, we compute the Poincar\'e series as follows:
\begin{align*}
\sum_{g \in G} a^{-sd(o,g(o))}&=\sum_{n \in \Z} \sum_{[g] \in \langle h \rangle \backslash G} a^{-sd(o,h^ng(o))}\\
&\leq a^{2s(\ell_h+r+4\delta)}\,\sum_{n \in \Z} a^{-sd(o,h^n(o))} \sum_{[g] \in \langle h \rangle \backslash G} a^{-s d(o,[g](o))}.
\end{align*}
Here, $\sum_{n \in \Z} a^{-sd(o,h^n(o))}$ ($s>0$) converges because $\beta$ is a quasi-geodesic that satisfies 
$$
d(o,h^n(o)) \geq \lambda^{-1}\ell_h n-c.
$$
Hence, by setting $A_h(s)=a^{2s(\ell_h+r+4\delta)}\,\sum_{n \in \Z} a^{-sd(o,h^n(o))}$, we obtain the assertion.
\end{proof}

\noindent
{\it Proof of Theorem \ref{main3}.}
Choose a hyperbolic element $h \in \Gamma$ and fix it. Concerning the map $\iota_h$ in Lemma \ref{finite-to-one}, we note that
\begin{align*}
&\quad \ d(o,\iota_h([g])(0))=d(o,g^{-1}hg(o))\\
&\leq d(o,g^{-1}(o))+d(g^{-1}(o),g^{-1}h(o))+d(g^{-1}h(o),g^{-1}hg(o))\\
&=2d(o,g(o))+d(o,h(o)).
\end{align*}
As this is still valid after replacing $g$ with $h^n g$ $(n \in \Z)$, we see that
$$
d(o,\iota_h([g])(o)) \leq 2d(o,[g](o))+d(o,h(o)).
$$
It follows that
$$
a^{-sd(o,[g](o))} \leq a^{s d(o,h(o))/2} \cdot a^{-s d(o,\iota_h([g])(o))/2}.
$$
Taking the sum over $[g] \in \langle h \rangle \backslash G$, we obtain
$$
\sum_{[g] \in \langle h \rangle \backslash G} a^{-sd(o,[g](o))} \leq a^{s \ell_h/2} 
\sum_{[g] \in \langle h \rangle \backslash G} a^{-s d(o,\iota_h([g])(o))/2}.
$$

Concerning the right-hand side in the above inequality, Lemma \ref{finite-to-one} implies that
$$
\sum_{[g] \in \langle h \rangle \backslash G} a^{-s d(o,\iota_h([g])(o))/2} \leq 
k_h \sum_{\gamma \in \Gamma} a^{-s d(o,\gamma(o))/2}.
$$
Concerning the left-hand side, Lemma \ref{essential} implies that
$$
\sum _{g \in G} a^{-sd(o,g(o))} \leq A_h(s) \sum_{[g] \in \langle h \rangle \backslash G} a^{-sd(o,[g](o))}.
$$
Combining these inequalities, we finally obtain the following estimate:
$$
\sum _{g \in G} a^{-sd(o,g(o))} \leq A_h(s) a^{s \ell_h/2} k_h  \sum_{\gamma \in \Gamma} a^{-s d(o,\gamma(o))/2}.
$$

Now we put $s=2e_a(\Gamma)+\varepsilon$ for an arbitrary $\varepsilon>0$ and consider the final estimate just above.
The right-hand side converges; hence, so does the left-hand side. This shows that $e_a(G) \leq 2e_a(\Gamma)+\varepsilon$.
As $\varepsilon>0$ is arbitrary, we have $e_a(G) \leq 2e_a(\Gamma)$, which yields the first assertion of the
theorem. 

Next, we assume that $G$ is of divergence type. Then, we put $s=e_a(G)$ and consider the same inequality.
In this case, the left-hand side diverges; hence, so does the right-hand side. To prove the strict inequality, 
it is assumed toward a contradiction that $e_a(\Gamma)=e_a(G)/2$. As the exponent of the series on the right-hand side is $s/2=e_a(G)/2=e_a(\Gamma)$,
$\Gamma$ must be of divergence type under this assumption. Theorem \ref{normal} then implies that
$e_a(G)=e_a(\Gamma)$. This is possible only when $e_a(G)=e_a(\Gamma)=0$. However,
this contradicts the next claim, which we can also find in \cite[Corollaire 5.5]{Co}.
\qed

\begin{proposition}\label{positive}
The critical exponent $e_a(G)$ of a non-elementary discrete group $G \subset \Isom(X,d)$ is
strictly positive.
\end{proposition}

\begin{proof}
Let $h \in G$ be a hyperbolic element. Then, by the last part of the proof of Lemma \ref{essential}, we see that
$e_a(\langle h \rangle)=0$ and $\langle h \rangle$ is of divergence type.
As $\Lambda(G) \supsetneqq \Lambda(\langle h \rangle)$, Corollary \ref{inclusion} shows that
$e_a(G)>0$.
\end{proof}

\medskip
\noindent
{\it Note added in the revision.} 
Recently, we have found the following references that are closely
related to the results in this study. 
Das, Simmons, and Urba\'nski \cite[Theorem 1.4.1]{DSU} proved Theorem \ref{main1}
in a more general setting.
Arzhantseva and Cashen \cite{AC} proved a special case of Theorem \ref{i-lower}
for isometric actions not necessarily on the Gromov hyperbolic spaces.

\bigskip


\bigskip
\bigskip
\begin{thebibliography}{99}

\bibitem{AC}
G. N. Arzhantseva and C. Cashen,
{\it Cogrowth for group actions with strongly contracting elements},
Ergod. Th. \& Dynam. Sys.,
https://doi.org/10.1017/etds.2018.123

\bibitem{Co}
M. Coornaert,
{\it Mesures de Patterson--Sullivan sur le bord d'un espace hyperbolique au sens de Gromov},
Pacific J. Math. {\bf 159} (1993), 241--270.

\bibitem{CDP}
M. Coornaert, T. Delzant and A. Papadopoulos,
{\it G\'eom\'etrie et th\'eorie des groupes}, 
Lecture Notes in Math. vol. 1441, Springer, 1990.

\bibitem{DSU}
T. Das, D. Simmons and M. Urba\'nski,
{\it Geometry and dynamics in Gromov hyperbolic metric spaces: with an emphasis on non-proper settings},
Mathematical Surveys and Monographs vol. 218, American Mathematical Society, 2017.

\bibitem{FS}
K. Falk and B. O. Stratmann, 
{\it Remarks on Hausdorff dimensions for transient limit sets of Kleinian groups}, 
Tohoku Math. J. {\bf 56} (2004), 571--582.

\bibitem{F}
H. Federer,
{\it Geometric measure theory},
Classics in Math., Springer, 1996.

\bibitem{J}
J. Jaerisch,
{\it A lower bound for the exponent of convergence of normal subgroups of Kleinian groups}, 
J. Geom. Anal. {\bf 25} (2015), 289--305.

\bibitem{MY1}
K. Matsuzaki and Y. Yabuki,
{\it The Patterson--Sullivan measure and proper conjugation for Kleinian groups of divergence type}, 
Ergod. Th. \& Dynam. Sys. {\bf 29} (2009), 657--665.

\bibitem{MY2}
K. Matsuzaki and Y. Yabuki,
{\it No proper conjugation for quasiconvex cocompact groups of
Gromov hyperbolic spaces}, In the tradition of Ahlfors and Bers VI,
Contemporary Math.
Vol. 590, pp. 125--136, American Mathematical Society, 2013.

\bibitem{N}
P. Nicholls, {\it The ergodic theory of discrete groups},
London Math. Soc. Lecture Note Series vol. 143, Cambridge Univ. Press, 1989.

\bibitem{O}
K. Ohshika, {\it Discrete groups},
Translation of Math. Monographs vol. 207, American Mathematical Society, 2001.

\bibitem{P}
M. Peign\'e,
{\it Autour de l'exposant de Poincar\'e d'un groupe kleinien},
G\'eom\'etrie ergodique, Monogr. Enseign. Math. vol. 43, pp. 25--59, Enseignement Math., 2013. 

\bibitem{R1}
T. Roblin,
{\it Ergodicit\'e et \'equidistribution en courbure n\'egative}, M\'em. Soc. Math. Fr. (N.S.) No. 95 (2003).

\bibitem{R2}
T. Roblin,
{\it Un th\'eor\`eme de Fatou pour les densit\'es conformes avec applications aux rev\^etements galoisiens en courbure n\'egative},
Israel J. Math. {\bf 147} (2005), 333--357.

\bibitem{RT}
T. Roblin and S. Tapie, 
{\it Exposants critiques et moyennabilit\'e},
G\'eom\'etrie ergodique, Monogr. Enseign. Math. vol. 43, pp. 61--92, Enseignement Math., 2013. 

\bibitem{S1}
D. Sullivan, {\it The density at infinity of a discrete group of hyperbolic motions},
Inst. Hautes \'Etudes Sci. Publ. Math. {\bf 50} (1979), 171--202.

\bibitem{S2}
D. Sullivan, {\it Discrete conformal groups and measurable dynamics}, 
Bull. Amer. Math. Soc. {\bf 6} (1982), 57--73.

\bibitem{T1}
P. Tukia,
{\it The Poincar\'e series and the conformal measure of conical and Myrberg limit points}, 
J. Anal. Math. {\bf 62} (1994), 241--259. 

\bibitem{T2}
P. Tukia,
{\it Convergence groups and Gromov's metric hyperbolic spaces},
New Zealand J. Math. {\bf 23} (1994), 157--187.

\end{thebibliography}
\end{document}